\documentclass{article}

\usepackage{multicol}
\usepackage{amsmath,amsfonts,amsthm,amssymb}
\usepackage{setspace}
\usepackage{graphicx}
\usepackage{tikz}
\usepackage{enumitem}
\usepackage{verbatim}
\usepackage{color}
\usepackage{xcolor}
\usepackage[normalem]{ulem}
\usepackage{comment,nicefrac}
\usepackage{url}

\newtheorem{thm}{Theorem}

\newtheorem{cnj}[thm]{Conjecture}
\newtheorem{cor}[thm]{Corollary}
\newtheorem{fct}[thm]{Fact}
\newtheorem{lem}[thm]{Lemma}

\def\b{{\beta}}

\def\f{{\phi}}

\def\i{{\iota}}
\def\k{{\kappa}}
\def\l{{\lambda}}

\def\r{{\rho}}
\def\s{{\sigma}}

\def\w{{\omega}}

\def\x{{\chi}}

\def\cC{{\cal C}}

\def\cM{{\cal M}}

\def\cP{{\cal P}}

\def\zN{{\mathbb N}}

\def\bz{{\mathbf{0}}}
\def\b1{{\mathbf{1}}}

\def\deg{{\sf deg}}
\def\diam{{\sf diam}}
\def\dist{{\sf dist}}
\def\ecc{{\sf ecc}}

\def\lev{{\sf level}}

\def\wt{{\sf wt}}

\def\rar{{\rightarrow}}

\def\ssq{{\hspace{0.02in} \square\hspace{0.02in}}}
\def\sse{{\subseteq}}

\definecolor{blu}{RGB}{  51,153,255}
\definecolor{brwn}{RGB}{140, 70, 20}
\definecolor{gren}{RGB}{  0,140, 10}

\newcommand{\pf}[1]{\textcolor{blue}{\sf{#1}}}
\newcommand{\gh}[1]{\textcolor{gren}{\sf{#1}}}
\newcommand{\mt}[1]{\textcolor{red}{\sf{#1}}}

\pgfmathsetmacro\Radi{6}
\pgfmathsetmacro\radi{3}

\begin{document}

\title{Cup Stacking in Graphs} 

\author{
Paul Fay
    \thanks{
        Department of Mathematics and Applied Mathematics, Virginia Commonwealth University, Richmond, Virginia, USA}
    \thanks{
        \texttt{fayph@vcu.edu}}
\and
Glenn Hurlbert 
    \footnotemark[1]
    \thanks{
        \texttt{ghurlbert@vcu.edu} (corresponding author), ORCID ID 0000-0003-2906-7770}
\and
Maya Tennant 
    \footnotemark[1]
    \thanks{
        \texttt{tennantm@vcu.edu}}
}

\date{}

\maketitle

\doublespacing

\begin{abstract}
Here we introduce a new game on graphs, called cup stacking, following a line of what can be considered as $0$-, $1$-, or $2$-person games such as chip firing, percolation, graph burning, zero forcing, cops and robbers, graph pebbling, and graph pegging, among others.
It can be more general, but the most basic scenario begins with a single cup on each vertex of a graph.
(This simplification coincides with an earlier game devised by Gordon Hamilton.)
For a vertex with $k$ cups on it we can move all its cups to a vertex at distance $k$ from it, provided the second vertex already has at least one cup on it.
The object is to stack all cups onto some pre-described target vertex.
We say that a graph is stackable if this can be accomplished for all possible target vertices.
    
In this paper we study cup stacking on many families of graphs, developing a characterization of stackability in graphs and using it to prove the stackability of complete graphs, paths, cycles, grids, the Petersen graph, many Kneser graphs, some trees, cubes of dimension up to 20, ``somewhat balanced'' complete $t$-partite graphs, and Hamiltonian diameter two graphs.
Additionally we use the Gallai-Edmonds Structure Theorem, the Edmonds Blossom Algorithm, and the Hungarian algorithm to devise a polynomial algorithm to decide if a diameter two graph is stackable.
    
Our proof that cubes up to dimension 20 are stackable uses Kleitman's Symmetric Chain Decomposition and the new result of Merino, M\"utze, and Namrata that all generalized Johnson graphs (excluding the Petersen graph) are Hamiltonian.
We conjecture that all cubes and higher-dimensional grids are stackable, and leave the reader with several open problems, questions, and generalizations.
\end{abstract}

\begin{quote}
\end{quote}

\newpage

\section{Introduction}
\label{s:Intro}

Cup stacking, also known as sport stacking or speed stacking, has become a popular sport involving the stacking and un-stacking of cups in various formations, in certain order, and in minimum time (see \cite{NYT,Stacker}).
Here we introduce a graph theoretic version of the stacking game.\footnote{We are indebted to one of the referees, who pointed us to a similar game, invented earlier by Hamilton \cite{Hamil}, called Frog Jumping; the special case of that game involving a single {\it lazy frog} coincides with the special case of our game involving the configuration of exactly one cup per vertex.
Subsequently, the Master's thesis \cite{Vesel} broadened Hamilton's path game to all graphs, explicitly studying the game on dandelions, which we define and discuss below.}

There are many varieties of what may be called ``moving things around in graphs'', including cops and robbers \cite{BonaNowa}, graph pebbling \cite{HurlKent}, zero forcing \cite{BBFHHSVV}, chip firing \cite{BjoLovSho}, graph burning \cite{Bonato}, and graph pegging \cite{HeKhMoWo}, among others.
Even network optimization and network flow can fall under in this umbrella, as can one of the world's oldest board games, Mancala \cite{JonTaaTon,Russ}.
In each scenario, we begin with a configuration of objects that are placed on the vertices of a graph $G$ (modeled as a nonnegative function on the vertices) and a set of target vertices to which we must move the objects.
Each scenario has different rules for moving objects along the edges of the graph --- movement may cost money, objects may be consumed in transit, edges may have capacities, or conditions on neighboring vertices may exist --- and sometimes the target(s) can move, controlled, say, by an adversary.
The objective may simply be to reach the target(s), or to do so at minimum cost or in minimum time, for example.
Our setup for cup stacking is as follows.

We will work exclusively with connected graphs.
Let $G$ be such a graph, with objects called {\it cups}, which are placed on the vertices of $G$.
A {\it configuration} $C$ on $G$ is a function $C:V(G)\rar\zN$ that encodes the number $C(v)$ of cups on each vertex $v$.
A {\it cup stacking move} from vertex $u$ to vertex $v$, denoted $u\mapsto v$, can be made if both vertices have at least one cup and $\dist(u,v)=C(u)$, and consists of moving all the cups from $u$ onto $v$.
Because of the distance condition, we know that $u\mapsto v$ can be thought of as carrying the $C(u)$ cups from $u$ to $v$ along a $uv$-{\it geodesic} (minimum length $uv$-path) $P(u,v)$ of length $C(u)$.
The resulting configuration $C'$ satisfies $C'(u)=0$, $C'(v)=C(u)+C(v)$, and $C'(w)=C(w)$ otherwise.
For a target vertex $r$, we say that $G$ is $(C,r)$-{\it stackable} if it is possible to stack every cup on $r$; it is {\it $C$-stackable} if it $(C,r)$-stackable for every $r$ with $C(r)>0$.

We define the configuration $\b1$ to have one cup on each vertex of $G$ (we will use this notation for any size graph).
When $C=\b1$, we will drop $C$ from the notation to say $r$-stackable and stackable, respectively.
The focus of this paper is to study the case that $C=\b1$.
However, it is important to note that, in order to do so, we still need to consider the general case (e.g. see Figure \ref{fig:StackPart}).

A simple observation when $C=\b1$, for example, is that the tree $K_{1,m}$ is $r$-stackable if and only if $m\le 2$ or $r$ is not a leaf, and thus $K_{1,m}$ is stackable if and only if $m\le 2$.
More generally we have the following fact.

\begin{fct}
\label{f:Dom}
If $r$ is a dominating vertex in a graph $G$ on $n$ vertices then $G$ is $r$-stackable. \end{fct}

Equally simple is the following.

\begin{fct}
\label{f:VTrans}
If $G$ is a vertex transitive graph and $r\in V(G)$ then $G$ is stackable if and only if $G$ is $r$-stackable.
\end{fct}

Together, Facts \ref{f:Dom} and \ref{f:VTrans} imply, for example, that all complete graphs are stackable.
Interestingly, Veselovac \cite{Vesel} found an infinite family of trees that are not $r$-stackable for any $r$.

\subsection{Definitions}
\label{ss:Defs}

We mostly follow standard graph theory terminology and notation, viewing a graph $G=(V,E)$ as a set of {\it vertices} $V$ together with a set of {\it edges} $E$, which are unordered pairs of vertices.
If necessary to identify to which graph a set of vertices or edges belongs, we may write $V(G)$ and $E(G)$ for the vertices and edges of the graph $G$.
We say that vertices $u$ and $v$ are {\it adjacent} in $G$ if $uv\in E(G)$ (note that we write $uv$ as a simplification of $\{u,v\}$).
We tend to save subscripts to denote the number of vertices of a graph; for example, $P_n$ is the path on $n$ vertices and thus has {\it length} $n-1$.

The {\it open neighborhood} of a vertex $v$ is denoted $N(v)$ and consists of the set of vertices that are adjacent to $v$, and the {\it degree} of $v$ is denoted $\deg(v)$, and equals $|N(v)|$.
The {\it closed neighborhood} of $v$ is denoted $N[v]$ and equals $N(v)\cup\{v\}$.
(Each of these can be written with the subscript $G$ --- e.g. $N_G(v)$, etc. --- if necessary to avoid confusion.)
We say that $v$ is a {\it dominating vertex} in $G$ if $N[v]=V(G)$.
More generally, for any $S\subset V(G)$ we write $N(S) = \cup_{v\in S}N(s)$ and $N[S]=N(S)\cup S$; then $S$ is a {\it dominating set} if $N[S]=V(G)$.

The {\it distance} between two vertices $u$ and $v$ is denoted $\dist(u,v)$, and counts the length of the shortest path with endpoints $u$ and $v$.
If necessary to distinguish distances in different graphs containing $u$ and $v$, we will write $\dist_G(u,v)$ to identify distance in the graph $G$.
The {\it diameter} of a graph $G$ is denoted $\diam(G)$ and counts the largest distance between two vertices of $G$; i.e. $\diam(G) = \max_{u,v}\dist_G(u,v)$.

The {\it Cartesian product} of two graphs $G$ and $H$ is denoted $G\ssq H$ and has vertex set $\{(x,y)\mid x\in V(G), y\in V(H)\}$ and edge set $\{(x,y)(x,y')\mid yy'\in E(H)\}\cup \{(x,y)(x',y)\mid xx'\in V(G)\}$.
The $d$-{\it dimensional cube} $Q^d$ (or $d$-{\it cube}) has vertex set all binary $d$-tuples $(a_1,\ldots,a_d)$ and edge set all pairs of vertices whose coordinates differ in exactly one position.
Thus $Q^0$ is a single vertex, $Q^1$ is isomorphic to $K_2$ (the complete graph on 2 vertices), $Q^2$ is isomorphic to $C_4$ (the cycle on 4 vertices), and $Q^3$ looks like the bordering edges of a 3-dimensional box.
As can be readily seen, we can also write $Q^d = Q^1\ssq Q^{d-1}$ or, more generally, $Q^d = Q^p\ssq Q^q$ for any natural numbers for which $p+q=d$.
Furthermore, we can identify $V(Q^d)$ with the family of all subsets of $\{1,\ldots,d\}$ by associating the set $A$ with the $d$-tuple defined by $a_i=1$ if and only if $i\in A$.
In this way, every $uv\in E(Q^d)$ corresponds to a pair of sets $U$ and $V$ for which $|U\triangle V|=1$.
In general, we have $\dist_{Q^d}(u,v)=|U\triangle V|$; in particular we have $\dist(u,\bz)=|U|$,
where $\bz$ denotes the all-zeros vertex.

In this paper we will assume that all variables mentioned are nonnegative integers.

\subsection{Results}
\label{ss:Results}

Let $C$ be any configuration of cups on $G$.
We define the {\it size} of $C$ to be $|C|=\sum_w C(w)$.
A configuration $D$ is called a {\it sub-configuration} of $C$ if $D(w)\le C(w)$ for every vertex $w$.
For a target $r$ of $G$, a subgraph $H$ of $G$, and a sub-configuration $D$ of $C$, $H$ is called $(D,r)$-{\it feasible} if either $D=\emptyset$ or there is some vertex $s$ of $H$ such that $H$ is $(D,s)$-stackable and $\dist_G(s,r)=|D|$.

We say that two configurations $D_1$ and $D_2$ are {\it disjoint} (and write $D_1\cap D_2=\emptyset$) if, for every vertex $w$, either $D_1(w)=0$ or $D_2(w)=0$.
Additionally, we define the configuration $D=D_1+D_2$ by $D(v)=D_1(v)+D_2(v)$.
Thus we say that the set of configurations $\{D_1,\ldots,D_k\}$ {\it partitions} $D$ if they are pairwise disjoint and $\sum_{i=1}^kD_i=D$.
Furthermore, we define the configuration $C-r$ to be $(C-r)(r)=0$ and $(C-r)(v)=C(v)$ otherwise.
For a graph $G$ and configuration $C$ on $G$ with $C(r)>0$, let $\cP=\{(H_1,D_1),\ldots,(H_k,D_k)\}$ for some $k\ge 1$, where each $H_i$ is a subgraph of $G$ and each $D_i$ is a sub-configuration of $C$.
The family $\cP$ is called an $r$-{\it stacking partition} of $(G,C)$ (see Figure \ref{fig:StackPart}) if
\begin{enumerate}
    \item 
    \label{p:union}
    $V(H_1)\cup\cdots\cup V(H_k)=V(G)-\{r\}$,
    \item 
    \label{p:partition}
    $\{D_1,\cdots,D_k\}$ partitions $C-r$,
    \item
    \label{p:feasible}
    each $H_i$ is $(D_i,r)$-feasible (stacking onto $v_i$), and
    \item
    \label{p:distances}
    each $H_i$ preserves the distances of $G$ along the $v_i$-stacking moves in part (\ref{p:feasible}).
\end{enumerate}

\begin{figure}
\begin{center}
\begin{tikzpicture}[scale=.7]
\tikzstyle{every node}=[draw,circle,fill=black,minimum size=1pt,inner sep=2pt]
\draw node (r) [label=above: {$r$}] at (7,10) {};
\draw node (rr) [label=below: {$1$}] at (7,10) {};
\draw node (a) [label=below: {$1$}] at (7,4) {};
\draw node (b) [label=left: {$1$}] at (5,5) {};
\draw node (c) [label=left: {$1$}] at (6,7) {};
\draw node (d) [label=left: {$1$}] at (5,9) {};
\draw[line width=1.5pt,color=gren] (a) -- (b);
\draw[line width=1.5pt,color=gren] (b) -- (c);
\draw[line width=1.5pt,color=gren] (c) -- (d);
\draw node (e) [label=above: {$1$}] at (8,7) {};
\draw node (f) [label=right: {$1$}] at (9,5) {};
\draw node (g) [label=left: {$1$}] at (7,5.5) {};
\draw[line width=1.5pt,color=red] (e) -- (f);
\draw[line width=1.5pt,color=red] (f) -- (g);
\draw[line width=0.5pt] (g) -- (e);
\draw node (h) [label=above: {$1$}] at (11,9) {};
\draw node (i) [label=above: {$1$}] at (9,9) {};
\draw node (j) [label=above: {$1$}] at (8,7) {};
\draw node (k) [label=above: {$1$}] at (10,7) {};
\draw[line width=1.5pt,color=blue] (h) -- (i);
\draw[line width=1.5pt,color=blue] (i) -- (j);
\draw[line width=1.5pt,color=blue] (j) -- (k);
\draw[line width=0.5pt] (d) -- (r);
\draw[line width=0.5pt] (f) -- (a);
\draw[line width=0.5pt] (e) -- (c);
\draw[line width=0.5pt] (r) -- (i);
\end{tikzpicture}
$\qquad\qquad$
\begin{tikzpicture}[scale=.7]
\tikzstyle{every node}=[draw,circle,fill=black,minimum size=1pt,inner sep=2pt]
\draw node (a) [label=below: {$1$}] at (6,4) {};
\draw node (aa) [label=right: {$v_1$}] at (6,4) {};
\draw node (b) [label=left: {$1$}] at (4,5) {};
\draw node (c) [label=left: {$1$}] at (5,7) {};
\draw node (d) [label=left: {$1$}] at (4,9) {};
\draw node (e) [label=above: {$1$}] at (8,7) {};
\draw node (f) [label=right: {$1$}] at (9,5) {};
\draw node (ff) [label=below: {$v_2$}] at (9,5) {};
\draw node (g) [label=left: {$1$}] at (7,5.5) {};
\draw node (h) [label=above: {$1$}] at (13,9) {};
\draw node (i) [label=above: {$1$}] at (11,9) {};
\draw node (j) [label=above: {$0$}] at (10,7) {};
\draw node (k) [label=above: {$1$}] at (12,7) {};
\draw node (kk) [label=right: {$v_3$}] at (12,7) {};
\draw[line width=1.5pt,color=gren] (a) -- (b) -- (c) -- (d);
\draw[line width=1.5pt,color=red] (e) -- (f) -- (g);
\draw[line width=1.5pt,color=blue] (h) -- (i) -- (j) -- (k);
\end{tikzpicture}
\caption{A graph and configuration $(G,C)$, left, with an $r$-stacking partition $\{{\color{gren} (H_1,D_1)},{\color{red} (H_2,D_2)},{\color{blue} (H_3,D_3)}\}$, right.}
\label{fig:StackPart}
\end{center}
\end{figure}
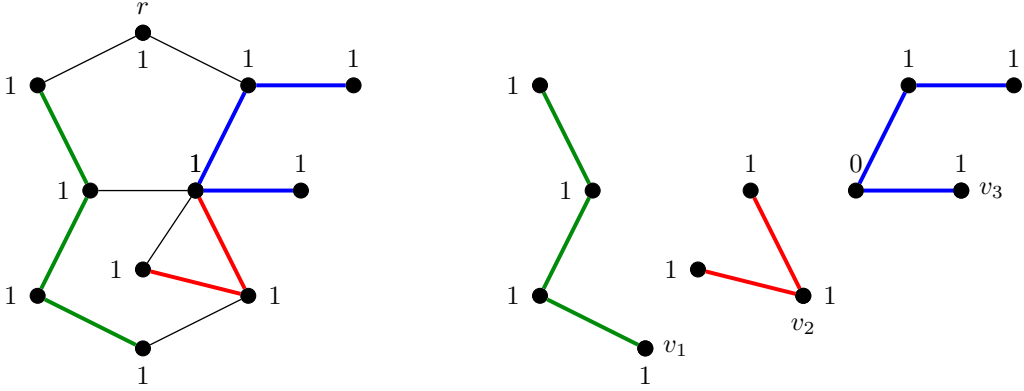

Our first result is the following characterization lemma, which we will prove in Section \ref{ss:StackingLemma}.

\begin{lem}[Stacking Partition Lemma]
\label{l:Partition}
Let $G$ be a graph, $C$ be any configuration of cups on $G$, and $r$ be any target vertex.
Then $G$ is $(C,r)$-stackable if and only if $(G,C)$ has an $r$-stacking partition.
\end{lem}

As a simple corollary we obtain the following fact, which we will prove in Section \ref{ss:PathsCyclesSpiders}.
First, a vertex of degree at least three in a graph is called a {\it split} vertex.
A tree with exactly one split vertex is called a {\it spider}.
We denote $S(\ell_1,\ldots,\ell_k)$ to be the spider with $k\ge 3$ paths emanating from the split vertex (also called a {\it root}), having lengths $\ell_1,\ldots,\ell_k$, for a total of $\ell_1+\cdots+\ell_k+1$ vertices.
As a slight abuse, we also consider a path to be a spider in the case that the root (now not a split vertex --- $k=2$) is not one of its endpoints; that is, the path $v_1\cdots v_n$ can be written as $S(i-1,n-i)$ with root $v_i$ when $1<i<n$.

\begin{fct}
\label{f:PathsCyclesSpiders}
Paths and cycles are stackable, and every spider with split vertex $r$ is $r$-stackable.
\end{fct}

The Fact \ref{f:PathsCyclesSpiders} case of paths was first proved in \cite{Hamil}, using a somewhat complicated induction proof.
The Fact \ref{f:PathsCyclesSpiders} case of spiders was first proved in \cite{Vesel}; in fact Veselovac proved more.
It is easy to see that the spider $S(1,1,\ldots)$ is $r$-stackable if and only if $r$ is the split vertex.

\begin{thm}
\cite{Vesel}
For all $k\ge 1$ and all $\ell_i>1$, $S(\ell_1,\ldots,\ell_k)$ is stackable.
\end{thm}

Left open are the cases for which $k\ge 3$, some $\ell_i=1$, and some $\ell_j>1$.
In \cite{Vesel} is found a complete solution for the subcase of having a unique $j$ with $\ell_j>1$; the statement of the theorem that describes for which vertices $r$ the particular graph is $r$-stackable is not describable succinctly, so we refer the reader to \cite{Vesel}.

We also use Lemma \ref{l:Partition} to prove the following two theorems, which will be proven in Sections \ref{ss:Ecc2} and \ref{ss:Cubes}.

\begin{thm}
\label{t:Diam2}
There is a polynomial algorithm to determine if a diameter two graph is $r$-stackable.
\end{thm}

The proof of Theorem \ref{t:Diam2} uses the Gallai-Edmonds Structure Theorem \ref{t:GalEd} (Section \ref{ss:Ecc2}) from matching theory.

\begin{thm}
\label{t:Grids}
The grid $P_m\ssq P_k$ is stackable for all $m$ and $k$.
\end{thm}

Furthermore, Lemma \ref{l:Partition} also yields stackability results for complete $t$-partite graphs and for Johnson graphs (which generalize Kneser graphs, which includes the Petersen graph), which we state in Section \ref{ss:Ecc2} as well.

\begin{thm}
\label{t:Cubes}
The $d$-dimensional cube $Q^d$ is stackable for all $d\le 20$.
\end{thm}

The proof of Theorem \ref{t:Cubes} uses the new result (Theorem \ref{t:JohnsonHam}) of Merino, M\"{u}tze, and Namrata that generalized Johnson graphs are Hamiltonian and the classic Symmetric Chain Decomposition Theorem for cubes (Theorem \ref{t:SCD}), both found in Section \ref{ss:Cubes}.

\section{Proofs}
\label{s:Proofs}

\subsection{Proof of the Stacking Partition Lemma}
\label{ss:StackingLemma}

\noindent
{\it Proof of Lemma \ref{l:Partition}.}
We begin by assuming that $(G,C)$ has an $r$-stacking partition $\cP=\{(H_1,D_1),\ldots,$ $(H_k,D_k)\}$.
We will show that every cup of $C$ on $v\in G-r$ can be placed on $r$.
By property \ref{p:partition} there is a unique $i$ such that $v\in D_i$.
By property \ref{p:feasible} there is some vertex $s$ of $H_i$ such that all the cups of $D_i$ can be placed on $s$ and $\dist(s,r)=|D_i|$.
Hence we can move all those cups from $s$ to $r$; i.e. $G$ is $(C,r)$-stackable.

Now we suppose that $G$ is $(C,r)$-stackable.
Consider the sequence of moves that place all the cups on $r$.
Let $s_1\ldots,s_m$ be the vertices from which cups move onto $r$ and let $X_i$ be the set of cups moved from $s_i$ to $r$.
For each $i$, let $H_i$ be the union of all geodesics $P(x,y)$ used in moving all the cups in $X_i$ from their original locations to $s_i$, and let $D_i$ count the number of cups of $X_i$ originally on each vertex of $H_i$.
Next, if $\cup_{i=1}^m H_i=G-r$ then set $k=m$; otherwise, let $(G-r)-\cup_{i=1}^m H_i = \{u_1,\ldots,u_t\}$, $k=m+t$, and $H_{m+i}=\{u_i\}$ for $1\le i\le t$, with $D_{m+i}=0$ on $H_{m+i}$.
Finally, define $\cP=\{(H_1,D_1),\ldots,(H_k,D_k)\}$ --- it is clear that $\cP$ satisfies each of the four properties of an $r$-stacking partition.
\hfill $\Box$

\subsection{Stackability of paths, cycles, and spiders}
\label{ss:PathsCyclesSpiders}
In this section we will be assuming that $C=\b1$. 
\bigskip

\noindent
{\it Proof of Fact \ref{f:PathsCyclesSpiders}.}
We prove that $P_n$ is stackable by induction.
The case $n=1$ is trivial, so we assume that $n\ge 2$ and that $P_k$ is stackable for all $1\le k<n$.
Let $r=v_i$ be any target vertex of $P_n$, with $1\le i\le n$. 
Partition $P_n$ into two subgraphs $H_1=\{v_j\in V(P_n):1\leq j\leq i\}$ and $H_2=\{v_j\in V(P_n): i\leq j\leq n\}$. 
Note that $H_1\cong P_i$ and $H_2\cong P_{n-i+1}$.
By induction $\cP=\{(H_1,\b1),(H_2,\b1)\}$ is an $r$-stacking partition of $(P_n,\b1)$, and so $P_n$ is $r$-stackable by Lemma \ref{l:Partition}.

Now consider the cycle $C_n$ and any target vertex $r$.
By symmetry, we may assume that $r=v_i$, where $i=\lfloor n/2\rfloor$.
Define $H_1$ and $H_2$ as above and notice that, for each $j\in\{1,2\}$, the distance between vertices of $H_j$ are the same in $H_j$ as they are in $C_n$.
Hence $\cP=\{(H_1,\b1),(H_2,\b1)\}$ is an $r$-stacking partition of $(C_n,\b1)$, and so $C_n$ is $r$-stackable by Lemma \ref{l:Partition}.

Finally, let $T$ be an $n$-vertex spider with root $r$.
Then $T-r$ is a disjoint union of paths $H_{n_1}, \ldots, H_{n_k}$, where $n_1+\cdots+n_k=n-1$.
Then $\{(H_1,\b1),\ldots,(H_k,\b1)\}$ is an $r$-stacking partition of $(T,\b1)$ because paths are stackable, and so $T$ is $r$-stackable by Lemma \ref{l:Partition}.
\hfill $\Box$
\bigskip

\subsection{An $r$-stacking algorithm for eccentricity 2 vertices}
\label{ss:Ecc2}

We begin this section with some definitions leading up to the Gallai-Edmonds Structure Theorem.
For a graph $G=(V,E)$ and a subset $U\sse V$ define the {\it neighborhood of $U$ in} $V-U$, denoted $N(U)$, to be the set of vertices in $V-U$ that are adjacent to some vertex of $U$.
Let $\cM=\cM(G)$ be the set of all maximum matchings of $G$.
A vertex $v$ is called {\it inessential} if it is not saturated by some $M\in\cM$, and let $I=I(G)$ be the set of all inessential vertices of $G$.
Finally set $A=A(G):=N(I)$ and $Z=Z(G):=V-I-A$; thus $\{I,A,Z\}$ partitions $V$ --- we call the triple $(I,A,Z)$ the {\it Gallai-Edmonds partition} of $G$.
For ease of notation, we will also use $I$, $A$, and $Z$ to denote the subgraphs of $G$ induced by those vertices.
Furthermore, say that a connected graph $G$ is {\it factor-critical} if the removal of any vertex yields a graph having a perfect matching, and that a matching in $G$ is {\it near-perfect} if it saturates all but one vertex of $G$.

\begin{thm}[Gallai-Edmonds Structure Theorem]
\cite{Edmonds,GallaiKrit,GallaiMax}
\label{t:GalEd}
Let $R$ be a graph and let $(I,A,Z)$ be the Gallai-Edmonds partition of $R$.
Then the following four properties hold.
\begin{enumerate}
    \item 
    \label{ti:IFactCrit}
    Each component of $I$ is factor-critical.
    \item
    \label{ti:NbrGalEd}
    Every subset $X\sse A$ has neighbors in at least $|X|+1$ components of $I$.
    \item
    \label{ti:MaxGalEd} 
    Every maximum matching in $R$ consists of
    \begin{enumerate}
        \item 
        \label{tia:Imatch}
        a near-perfect matching of each component of $I$,
        \item 
        \label{tia:Zmatch}
        a perfect matching of $Z$, and 
        \item 
        \label{tia:Amatch}
        edges from all vertices in $A$ to distinct components of $I$.
    \end{enumerate}
    \item 
    If $I$ has $k$ components, then the number of edges in any maximum matching in $R$ equals $(|V|-k+|A|)/2$.
\end{enumerate}
Moreover, the partition $(I,A,Z)$ can be found in polynomial time.
\end{thm}

We begin be proving an $r$-stacking equivalence for eccentricity two vertices in terms of matchings.

\begin{lem}
\label{l:Char}
Let $G$ be a graph with vertex $r$ having eccentricity 2.
Then $G$ is $r$-stackable if and only if $G$ has a matching that saturates $N_2(r)$.
\end{lem}

\begin{proof}
Let $G$ be a graph with vertex $r$ having eccentricity 2.

Suppose that $n$-vertex $G$ has a matching $M=\{e_1,\ldots,e_k\}$ that saturates $N_2(r)$.
Label the vertices of each $e_i$ by $x_i$ and $y_i$, with $y_i\in N_2(r)$, and define $H_i=e_i$ and $D_i=\b1$ for $1\le i\le k$.
Then $|D_i|=2$ and each $H_i$ is $y_i$-feasible with $\dist(y_i,r)=2$.
Now label the $M$-unsaturated vertices of $G-r$ by $z_1,\ldots z_j$, where $j=n-1-2k$, and define $H_{k+i}=z_i$ and $D_{k+i}=\b1$ for all $1\le i\le j$.
Since each $z_i\in N(r)$ and $|D_{k+1}|=1$, $H_{k+i}$ is trivially $z_i$-feasible with $\dist(z_i,r)=1$.
Hence $\cP=\{(H_1,D_1),\ldots,(H_{k+j},D_{k+j})\}$ is an $r$-stacking partition, which implies by Lemma \ref{l:Partition} that $G$ is $r$-stackable.

Now suppose that $G$ is $r$-stackable; by Lemma \ref{l:Partition} $G$ has an $r$-stacking partition $\cP=\{(H_1,D_1),\ldots,$ $(H_k,D_k)\}$; let $\{v_1,\ldots,v_k\}$ be the set of vertices for which each $H_i$ is $(D_i,v_i)$-stackable.
Because $\dist(u,r)\le 2$ for all $r$ we have $|D_i|\le 2$ for all $1\le i\le k$.
Consider any vertex $v\in N_2(r)$, and suppose that $v\in H_i$ for some $i$  (which we can denote by $i(v)$ as necessary to emphasize its dependence on $v$) with $D_i(v)=1$.
If $v=v_i$ then there is a unique vertex $u_i$ adjacent to $v$ such that the cup on $u_i$ is moved to $v_i$.
If $v\not=v_i$ then set $u_i=v$ and note that the cup on $u_i$ is moved to $v_i$.
In either case label the edge $e_i=u_iv_i$.
Now the set of edges $M=\{e_i\mid v\in N_2(r), i=i(v)\}$ (ignoring duplicates that occur when both $u_i, v_i\in N_2(r)$) forms a matching that saturates $N_2(r)$.
\end{proof}

Lemma \ref{l:Char} yields the following five corollaries.

\begin{cor}
\label{c:d2Trees}
Let $r$ be a vertex of eccentricity 2 in a tree $T$.
Then $T$ is $r$-stackable if and only if $T$ is a spider.
\end{cor}

\begin{proof}
Let $T$ be a tree with vertex $r$ having eccentricity 2. 
Note that, because $\ecc(r)=2$, each $w\in N_2(r)$ is a leaf.
Thus each such leaf can be matched simultaneously to its unique neighbor in $N_1(r)$ if and only if each $v\in N_1(r)$ has $\deg(v)\le 2$; that is, if and only if $T$ is a spider.
The result then follows from Lemma \ref{l:Char}.

\end{proof}

\begin{cor}
\label{c:CompBip}
Let $G=K_{a_1,\ldots,a_t}$ denote the complete $t$-partite graph on $n$ vertices with parts $A_1,\ldots,A_t$ having sizes $|A_i|=a_i$, and let $r$ be a vertex in part $A_i$.
Then $G$ is $r$-stackable if and only if $a_i\leq (n+1)/2$.
Consequently, $G$ is stackable if and only if $a_i\leq (n+1)/2$ for every $i$.
\end{cor}

\begin{proof}
Let $G$ be a complete $t$-partite graph with partite sets $A_1,\ldots,A_t$ having sizes $|A_i|=a_i$. 
Notice that $G$ has diameter 2 and so, by utilizing Lemma \ref{l:Char}, $G$ is $r$-stackable if and only if $G$ has a matching that saturates $N_2(r)=A_i-\{r\}$; i.e. there exists a matching $M$ that saturates $N_2(r)$. 
This happens if and only if $n-a_i = |V(G)-A_i| \ge |M| = |N_2(r)|=a_i-1$, which is equivalent to $a_i\le (n+1)/2$.  
\end{proof}

\begin{cor}
\label{c:Pete}
The Petersen graph is stackable.
\end{cor}

\begin{proof}
Consider the Petersen graph $G_P$. 
By symmetry, we only need to consider one target vertex $r$. 
Since $G_P$ has diameter 2, by Lemma \ref{l:Char}, $G_P$ is stackable if and only if $N_2(r)$ has a saturating matching. 
The matching in Figure \ref{fig:Pete}, saturates $N_2(r)$.
\end{proof}

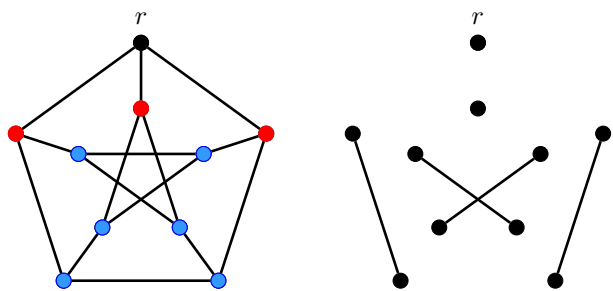
\begin{figure}[hb]
\begin{center}
\begin{tikzpicture}[scale=.35]
\tikzstyle{every node}=[draw,circle,fill=black,minimum size=1pt,inner sep=2pt]

    \begin{scope}[rotate=90]
        \foreach \x/\y in {0/1,72/2,144/3,216/4,288/5}{
            \draw node (\y) at (canvas polar cs: radius=2.5cm,angle=\x){};
            }
        \foreach \x/\y in {0/6,72/7,144/8,216/9,288/10}{
            \draw node (\y) at (canvas polar cs: radius=5cm,angle=\x){};
            }
    \end{scope}

    \foreach \x/\y in {1/6,2/7,3/8,4/9,5/10}{
        \draw[line width=1.0pt, color=black] (\x) -- (\y);
        }

    \foreach \x/\y in {1/3,2/4,3/5,4/1,5/2}{
        \draw[line width=1.0pt, color=black] (\x) -- (\y);
        }

    \foreach \x/\y in {6/7,7/8,8/9,9/10,10/6}{
        \draw[line width=1.0pt, color=black] (\x) -- (\y);
        }
    \draw node [fill=black,label=above: {$r$}] at (6) {};
    
    \foreach \x in {1,7,10}{
        \draw node [red,fill=red] at (\x) {};
        }
    \foreach \x in {2,...,5,8,9}{
        \draw node [blue,fill=blu] at (\x) {};
        }
\end{tikzpicture}
$\qquad$
\begin{tikzpicture}[scale=.35]
\tikzstyle{every node}=[draw,circle,fill=black,minimum size=1pt,inner sep=2pt]

    \begin{scope}[rotate=90]
        \foreach \x/\y in {0/1,72/2,144/3,216/4,288/5}{
            \draw node (\y) at (canvas polar cs: radius=2.5cm,angle=\x){};
            }
        \foreach \x/\y in {0/6,72/7,144/8,216/9,288/10}{
            \draw node (\y) at (canvas polar cs: radius=5cm,angle=\x){};
            }
    \end{scope}
    \foreach \x/\y in {2/4,3/5,7/8,9/10}{
        \draw[line width=1.0pt, color=black] (\x) -- (\y);
        }
    \draw node [label=above: {$r$}] at (6) {};
\end{tikzpicture}
\caption{The Petersen graph, showing ${\color{red} N_1(r)}$ and ${\color{blue} N_2(r)}$, left, and an $N_2(r)$ saturating matching, right.}
\label{fig:Pete}
\end{center}
\end{figure}

In fact, we can show that infinitely many Kneser graphs are stackable.
We write $[m] = \{1,\ldots,m\}$ and define $\binom{[m]}{k}$ to be the set of all subsets of $[m]$ of size exactly $k$. 
For $m\ge 2k+1\ge 5$, the {\it Kneser graph} $K(m,k)$ has vertex set $\binom{[m]}{k}$, with edges between disjoint pairs.
For example, the Petersen graph is $K(5,2)$.
We begin by observing a more general corollary to Lemma \ref{l:Char}.

\begin{cor}
\label{c:HamD2Stack}
Let $G$ be a graph with vertex $r$ having eccentricity 2.
If $G$ has a Hamiltonian path then $G$ is $r$-stackable.
\end{cor}

\begin{proof}
Consider a Hamiltonian path $P$ of $G$ and let $P-r$ be the disjoint union of paths $P_1$ and $P_2$, where $P_2$ is empty if $r$ is an endpoint of $P$.
For each $i$ let $v_i$ be the endpoint of $P_i$ that is a neighbor of $r$ in $P$.
If $n(P_i)$ is even then let $M_i$ be the perfect matching in $P_i$, while if $n(P_i)$ is odd then let $M_i$ be the perfect matching in $P_i-v_i$.
In both cases $M_i$ saturates $N_2(r)\cap V(P_i)$, and so $M_1\cup M_2$ saturates $N_2(r)$.
Hence $G$ is $r$-stackable by Lemma \ref{l:Char}.
\end{proof}

Notice that the hypothesis of Theorem \ref{c:HamD2Stack} contains diameter two Hamiltonian graphs.
With regard to Kneser graphs, we will use the following fact, which is a special case of Theorem \ref{t:JohnsonHam} below.

\begin{fct}
\label{f:KneserHam}
For all $m\ge 2k+1\ge 5$ except the case $(m,k)=(5,2)$ the Kneser graph $K(m,k)$ is Hamiltonian.
\end{fct}

\begin{cor}
\label{c:KneserStack}
For all $m\ge 3k-1\ge 5$ the Kneser graph $K(m,k)$ is stackable.
\end{cor}

\begin{proof}
The case $K(5,2)$ is proven in Theorem \ref{c:Pete}.
In all other cases, we have that $K(m,t)$ is Hamiltonian by Fact \ref{f:KneserHam}.
Moreover, since $m\ge 3t-1$, $K(m,t)$ has diameter two.
Indeed, if vertices $A$ and $B$ are not adjacent, then $|A\cup B|\le 2t-1$, and so there is some $C\ \sse\ [m]-(A\cup B)$ that is therefore a common neighbor of $A$ and $B$.
Hence Corollary \ref{c:HamD2Stack} implies that $K(m,t)$ is $r$-stackable.
Because Kneser graphs are vertex transitive, $K(m,t)$ is $r$-stackable for all $r$.
\end{proof}

It would be interesting to know whether or not all Kneser graphs are stackable, as Corollary \ref{c:KneserStack} leaves open the cases $5\le 2k+1\le m\le 3k-2$.

Now we are ready to prove the following theorem.

\begin{thm}
\label{t:Ecc2Algo}
Let $G$ be a graph with vertex $r$ having eccentricity 2.
Then it can be decided in polynomial time whether or not $G$ is $r$-stackable and, in the case that it is, an $r$-stacking can be found in polynomial time.
\end{thm}

\begin{proof} 
Let $G$ be a graph with configuration $C=\b1$ and target vertex $r$ having eccentricity 2.
We use Lemma \ref{l:Char} to determine whether or not $G$ has a matching that saturates $N_2(r)$.

Let $(I,A,Z)$ be a Gallai-Edmonds partition of $G-r$. 
By part \ref{tia:Zmatch} of Theorem \ref{t:GalEd}, $Z$ contains a perfect matching $M_Z$. 
For any set $S$ of vertices of $G$ and $i\in\{1,2\}$, let $S_i=S\cap N_i(r)$.
Then $G$ has a matching that saturates $N_2(r)$ if and only if $G-Z$ has a matching that saturates $I_2\cup A_2$.

Let $I^*$ denote the vertices defined by contracting each component of $I$ to a single vertex.
Note by part \ref{ti:NbrGalEd} of Theorem \ref{t:GalEd} we have $|I^*|>|A|$.
Then define the bipartite graph $B$ between $I^*$ and $A$ that arises from the contraction, ignoring the edges within $A$ itself.
By part \ref{tia:Amatch} of Theorem \ref{t:GalEd}, $B$ has an $A$-saturating matching.

Given any $A$-saturating matching $M_A$ in $B$, define $I^*_A$ to be the vertices of $I^*$ saturated by $M_A$.
For each edge $ax\in M_A$ we may choose, by definition, some $y_x$ in the component $I_x$ of $I$ represented by $x\in I^*_A$ such that $ay_x\in E(G)$; let $M'_A$ be the matching of all such edges.
By part \ref{ti:IFactCrit} of Theorem \ref{t:GalEd} we can find a perfect matching $M_x$ of $I_x-y_x$ for every $x\in I^*_A$.
For each $x\in I^*-I^*_A$ we can find a near-perfect matching $M_x$ of $I_x$ by part \ref{tia:Imatch} of Theorem \ref{t:GalEd}.
Thus $G-Z$ has a matching that saturates $I_2\cup A_2$ if and only if $G-Z$ has an $A$-saturating matching $M_A$ that saturates $I^*_2$ (i.e. $I^*_2\sse I^*_A$).

Finally, let $W$ be a set of $|I^*|-|A|$ new vertices, define $A^*=A\cup W$, and let $B^*$ denote the balanced complete bipartite graph $I^*\times A^*$ with edge weights $|I^*|+1$ on every edge of $B$ incident with $I^*_2$, $1$ on every remaining edge of $B$, and $0$ on every edge not in $B$. 
Then a matching in $B^*$ has weight at least $|I_2^*|(|I^*|+1)$ if and only if it saturates $I_2^*$.

Now let $M^*$ be a perfect matching in $B^*$ of maximum weight and let $M^*_A$ be the set of its edges incident with $A$.
It is easy to see that $M^*_A$ is an $A$-saturating matching (by part \ref{tia:Amatch} of Theorem \ref{t:GalEd}) that maximizes the number of saturated vertices of $I^*_2$.
Consequently, $G-Z$ has an $A$-saturating matching $M$ that saturates $I^*_2$ if and only if $M^*_A$ saturates $I^*_2$.

Define $M_G = M_Z \cup {M^*_A}' \cup_{x\in I^*} M_x$.
Then by Lemma \ref{l:Char} we see that $G$ is $r$-stackable if and only if $M_G$ saturates $N_2(r)$.
Notice that $M_G$ can be constructed in polynomial time.
Indeed, the Gallai-Edmonds partition $(I,A,Z)$ can be found in polynomial time by Theorem \ref{t:GalEd}, Edmonds' Blossom Algorithm finds $M_Z$ and each $M_x$ in polynomial time, and the Hungarian Algorithm finds $M^*_A$, and hence ${M^*_A}'$ in polynomial time.
\end{proof} 

\noindent {\it Proof of Theorem \ref{t:Diam2}.}
This follows from Theorem \ref{t:Ecc2Algo} because every vertex of a diameter two graph has eccentricity two.
\hfill $\Box$

\subsection{Stackability of grids}
\label{ss:Grids}

\begin{figure}[ht]
\begin{center}
\begin{tikzpicture}[scale=.4]
\tikzstyle{every node}=[draw,circle,fill=black,minimum size=1pt,inner sep=2pt]
\def \s {2}
\foreach \i in {1,...,8}
    \foreach \j in {1,...,8}
        \draw[line width=0.5pt, color=black] ({\i*\s},{\j*\s}) -- ({(\i+1)*\s},{\j*\s});
\foreach \i in {1,...,2}
    \foreach \j in {1,...,3}
        \draw[line width=1.5pt, color=blue] ({\i*\s},{\j*\s}) -- ({(\i+1)*\s},{\j*\s});
\draw[line width=1.5pt, color=blue] ({3*\s},{1*\s}) -- ({4*\s},{1*\s});
\draw node[color=blue] at ({3*\s},{1*\s}) {};
\draw node[color=blue] at ({3*\s},{2*\s}) {};
\draw node[color=blue] at ({2*\s},{3*\s}) {};
\foreach \i in {5,...,8}
    \foreach \j in {1,...,3}
        \draw[line width=1.5pt, color=brwn] ({\i*\s},{\j*\s}) -- ({(\i+1)*\s},{\j*\s});
\draw node[color=brwn] at ({6*\s},{1*\s}) {};
\draw node[color=brwn] at ({7*\s},{2*\s}) {};
\draw node[color=brwn] at ({8*\s},{3*\s}) {};
\foreach \i in {5,...,8}
    \foreach \j in {5,...,8}
        \draw[line width=1.5pt, color=gren] ({\i*\s},{\j*\s}) -- ({(\i+1)*\s},{\j*\s});
\draw node[color=gren] at ({8*\s},{5*\s}) {};
\draw node[color=gren] at ({7*\s},{6*\s}) {};
\draw node[color=gren] at ({6*\s},{7*\s}) {};
\draw node[color=gren] at ({5*\s},{8*\s}) {};
\foreach \i in {1,2,5,6,7,8}
    \foreach \j in {4,...,4}
        \draw[line width=1.5pt, color=orange] ({\i*\s},{\j*\s}) -- ({(\i+1)*\s},{\j*\s});
\draw node[color=orange] at ({1*\s},{4*\s}) {};
\draw node[color=orange] at ({9*\s},{4*\s}) {};
\foreach \i in {1,...,9}
    \foreach \j in {1,...,7}
        \draw[line width=0.5pt, color=black] ({\i*\s},{\j*\s}) -- ({\i*\s},{(\j+1)*\s});
\foreach \i in {1,...,3}
    \foreach \j in {5,...,7}
        \draw[line width=1.5pt, color=purple] ({\i*\s},{\j*\s}) -- ({\i*\s},{(\j+1)*\s});
\draw node[color=purple] at ({1*\s},{5*\s}) {};
\draw node[color=purple] at ({2*\s},{6*\s}) {};
\draw node[color=purple] at ({3*\s},{7*\s}) {};
\foreach \i in {4,...,4}
    \foreach \j in {2,5,6,7}
        \draw[line width=1.5pt, color=orange] ({\i*\s},{\j*\s}) -- ({\i*\s},{(\j+1)*\s});
\draw node[color=orange] at ({4*\s},{2*\s}) {};
\draw node[color=orange] at ({4*\s},{8*\s}) {};
\draw node [label={[label distance=1pt]225:$r$}] at ({4*\s},{4*\s}) {};
\end{tikzpicture}
\caption{An $r$-stacking partition for the grid $P_9\ssq P_8$.}
\label{fig:Grids}
\end{center}
\end{figure}
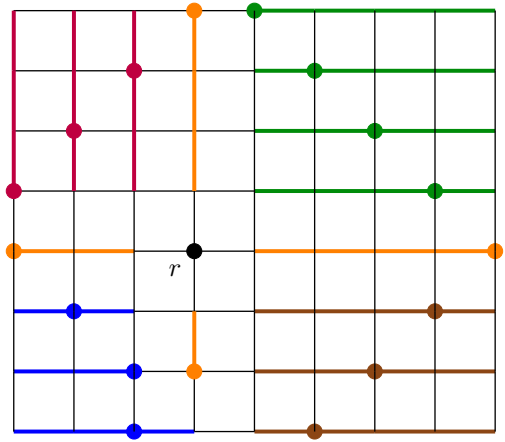

\noindent {\it Proof of Theorem \ref{t:Grids}.}
Let $G$ be the grid $P_m\ssq P_k$ with arbitrary $m,k$.
By symmetry we may assume that $m\ge k$; a grid with $m=k$ is called {\it square}.
We will first handle the case in which the target vertex $r$ is a corner of $G$.
Then we will reduce all other cases to this case.

Suppose that the target vertex $r$ is on a corner of $G$ --- say, $r$ is the bottom left corner --- with coordinates $r=(0,0)$.
Consider the case that $G$ is not square (for example, see the upper right green portion of Figure \ref{fig:Grids}). 
Partition $G$ into paths $R_0 = P_m-r$, $R^* = P_k-r$ emanating from $r$ (orange in Figure \ref{fig:Grids}) and paths $\{R_1,\ldots,R_k\}$ parallel to $P_m$ (green in Figure \ref{fig:Grids}), and let $r_i\in V(R_i)$ have coordinates $r_i=(m-i,i)$ for each $0\le i\le k$, with $r^*=(0,k)$.
Clearly, $\dist(r_i,r) = m = |R_i|$ for each $0\le i\le k$ and $\dist(r^*,r) = k = |R^*|$. 
Moreover, by Fact \ref{f:PathsCyclesSpiders}, each $R_i$ is $r_i$-stackable, and $R^*$ is $r^*$-stackable, and so $G$ is $r$-stackable by Lemma \ref{l:Partition}.

If $G$ is square ($m=k$), we use the same partition as before except that we shorten $R^* = P_{k-1}$ and lengthen $R_k \cong P_m$ (see the lower left blue portion of Figure \ref{fig:Grids}), with the corresponding changes $r_k=(1,k)$ and $r^*=(0,k-1)$, with $\dist(r_k,r) = k+1 = |R_k|$ and $\dist(r^*,r) = k-1 = |R^*|$.
Hence $G$ is $r$-stackable by Lemma \ref{l:Partition}.

For the cases in which $r$ is not a corner vertex of $G$, partition $G$ into two or four smaller grids such that $r$ is a corner of each (i.e. give $r$ the coordinates $(0,0)$; then the intersection of the grid with the four quadrants gives the partition). 
Each grid can be partitioned as above whether the quadrants form squares or not, with a slight modification, if necessary.
It may be that definitions of the orange paths $R^*$ in neighboring quadrants are not consistent, as is the case in the lower half of Figure \ref{fig:Grids}.
Thus we first partition the interior of each quadrant and subsequently define each $R^*$ to be whatever paths remain.
By Lemma \ref{l:Partition}, the grid $P_m\ssq P_k$ is stackable for all $m$ and $k$.
\hfill $\Box$

\subsection{Stackability of small cubes}
\label{ss:Cubes}

We begin by presenting two theorems that are used in the proof of Theorem \ref{t:Cubes}.
Define the {\it generalized Johnson graph} $J(m,k,s)$ to have vertex set $\binom{[m]}{k}$, with edges between sets sharing exactly $s$ elements.
For example, the Petersen graph is isomorphic to both $J(5,2,0)$ and $J(5,3,1)$, the {\it Odd graphs} are the family $\{J(2k+1,k,0)\mid k\ge 2\}$, the {\it Kneser graphs} are the family $\{J(m,k,0)\mid m\ge 2k+1\ge 5\}$, and the original {\it Johnson graphs} are the family $\{J(m,k,k-1)\mid m\ge k+1\ge 2\}$.
Settling a conjecture of \cite{Hurlbert} for Kneser graphs and proving even more, Merino, M\"utze, and Namrata \cite{MerMutNam} proved the following theorem.

\begin{thm}
\cite{MerMutNam}
\label{t:JohnsonHam}
Except for the Petersen graph, for all values with $s<k\le (m+s-1)/2$, the graph $J(m,k,s)$ is Hamiltonian.
\end{thm}

By using Corollary \ref{c:HamD2Stack} and Theorem \ref{t:JohnsonHam}, we obtain the following result as well.

\begin{cor}
\label{c:JohnsonStack}
Every diameter two Johnson graph is stackable.
\end{cor}

Next we write $2^{[n]}$ to denote the set of all subsets of $[n]$.
Then we define a {\it chain} to be a sequence $(S_1,\ldots,S_t)$ of subsets of $[n]$ such that $S_i\subseteq S_{i+1}$ for each $1\le i<t$; it is {\it saturated} if $|S_{i+1}|=|S_i|+1$ for each $1\le i<t$.
A saturated chain in $2^{[n]}$ is {\it symmetric} if $|S_1|+|S_t|=n$.
A family $\cC=\{C_1,\ldots,C_m\}$ of symmetric, saturated chains in $2^{[n]}$ is a {\it symmetric chain decomposition} if $\cC$ partitions $2^{[n]}$.

\begin{thm}
\cite{Kleitman}
\label{t:SCD}
For every $n$, $2^{[n]}$ has a symmetric chain decomposition.
\end{thm}

Theorem \ref{t:SCD} quickly gives rise to the following corollary.
Define $\binom{[n]}{\ge k}$ (resp. $\binom{[n]}{>k}$) to be the set of all subsets of $[n]$ of size at least (resp. greater than) $k$. 
Given a symmetric chain decomposition $\cC$, define the function $\f:\binom{[n]}{>n/2} \rar \binom{n}{\ge\lfloor n/2\rfloor}$ as follows.
For $X\in \binom{[n]}{>n/2}$ we find $C\in\cC$ such that $X\in C$.
Then find $Y\in C$ such that $|Y|=|X|-1$, and set $\f(X)=Y$.

\begin{cor}
\label{c:Inject}
For every $n$ the function $\f: \binom{[n]}{>n/2}\rar \binom{[n]}{\ge\lfloor n/2\rfloor}$ is an injection such that $|\f(X)|=|X|-1$ for all $X\in \binom{[n]}{>n/2}$.
\end{cor}

\begin{proof}
Because of Theorem \ref{t:SCD} such a symmetric chain decomposition $\cC$ of $2^{[n]}$ exists.
Because $\cC$ partitions $2^{[n]}$, for every $X\in \binom{[n]}{>n/2}$ the choice of $C\in\cC$, such that $X\in C$, exists and is unique.
Because $C$ is saturated and $|X|>n/2$, the choice of $Y\in C$, such that $|Y|=|X|-1\ge \lfloor n/2\rfloor$, exists and is unique.
Hence $\f$ is well-defined and is injective.
\end{proof}

Using the prior-mentioned identification between $2^{[d]}$ and $V(Q^d)$, we apply $\f$ to the appropriately associated vertices as well.
Recall in this association that each $u=(u_1,\ldots,u_d)\in V(Q^d)$ corresponds to the subset $U$ defined by $i\in U$ if and only if $u_i=1$.
From this we define the {\it weight} of $u$ by $\wt(u)=|U|$; i.e. $\wt(u)=\dist_{Q^d}(u,\bz)$.
Additionally, the product $Q^d = Q^p\ssq Q^q$ is a partitioning of $Q^d$ into $q$-cubes, where each vertex in $Q^p$ (subset of $[p]$) is replaced by a $q$-cube.
Hence we can label each such $q$-cube $A$ (i.e. not all $q$-cubes of $Q^d$, just those arising from this Cartesian-product-generated partition) uniquely and distinctly by its corresponding subset $S$ of $[p]$, and define the {\it level} of $A$ as $\lev(A)=|S|$.
Of course this value equals $\min_{a\in A}\wt(a)$.
We call the unique vertex of minimum (resp. maximum) weight in $A$ its {\it bottom} (resp. {\it top}) vertex.

\begin{lem}
\label{l:LevelkCubes}
Let $k\le 3$, $l\le 2^k$, and $d\ge l+k$, and suppose that if $k=3$ then $l\not=4$.
Then every level-$l$ k-cube in $Q^d$ is $\bz$-stackable.
In particular, $Q^d$ is stackable for all $d\le 6$.
\end{lem}

\begin{proof}
As noted in Fact \ref{f:VTrans}, $Q^d$ is stackable if and only if it is $\bz$-stackable.
When $k=0$ the statement is trivially true.

When $k=1$ we have $Q^k=K_2$, which we know to be stackable.
Thus we can stack it immediately onto $\bz$ if $l=0$ and stack it onto its weight-2 vertex before sending the 2 cups to $\bz$ if $l\in\{1,2\}$.
This implies that $Q^2$ and $Q^3$ are stackable since such 1-cubes form a stacking partition of them; that is, for $d\in\{2,3\}$ we have $Q^d=Q^{d-1}\ssq Q^1$, with each $Q^1$ having level at most 2.

When $k=2$ we have that any level-0 2-cube is $\bz$-stackable.
Also, any level-1 2-cube consists of two 1-cubes that each contain weight-2 vertices, and so two cups can be stacked onto each of those before being sent to $\bz$.
For $2\le l\le 4$ we can stack any level-$l$ 2-cube onto one of its weight-4 vertices before sending the 4 cups to $\bz$.
This implies that $Q^d$ is stackable for $d\le 6$ since we have $Q^d=Q^{d-2}\ssq Q^2$, with each $Q^2$ having level at most 4.

When $k=3$ we have that any 3-cube of level at most 3 is $\bz$-stackable because it can be partitioned into two 2-cubes of level at most 4.
If $5\le l\le 8$ then any level-$l$ 3-cube can be stacked onto one of its weight-8 vertices before sending all its cups to $\bz$.
This finishes the proof.
\end{proof}

\begin{figure}[ht]
\begin{center}
\begin{tikzpicture}[scale=.4]
\tikzstyle{every node}=[draw,circle,fill=black,minimum size=1pt,inner sep=2pt]
\def \r {-8}
\def \s {2.5}
\def \u {-2}
\def \v {2.5}
\def \x {1}
\def \y {2.5}
\def \w {4}
\def \z {2.5}
\foreach \j in {0,...,1}
    \foreach \k in {0,...,1}
        \foreach \h in {0,...,1}
            \draw[line width=1.0pt, color=black] ({0*\r+\j*\u+\k*\x+\h*\w},{0*\s+\j*\v+\k*\y+\h*\z}) -- ({\r+\j*\u+\k*\x+\h*\w},{\s+\j*\v+\k*\y+\h*\z});
\foreach \i in {0,...,1}
    \foreach \k in {0,...,1}
        \foreach \h in {0,...,1}
            \draw[line width=1.0pt, color=black] ({\i*\r+0*\u+\k*\x+\h*\w},{\i*\s+0*\v+\k*\y+\h*\z}) -- ({\i*\r+1*\u+\k*\x+\h*\w},{\i*\s+1*\v+\k*\y+\h*\z});
\foreach \i in {0,...,1}
    \foreach \j in {0,...,1}
        \foreach \h in {0,...,1}
            \draw[line width=1.0pt, color=black] ({\i*\r+\j*\u+0*\x+\h*\w},{\i*\s+\j*\v+0*\y+\h*\z}) -- ({\i*\r+\j*\u+\x+\h*\w},{\i*\s+\j*\v+\y+\h*\z});
\foreach \i in {0,...,1}
    \foreach \j in {0,...,1}
        \foreach \k in {0,...,1}
            \draw[line width=1.0pt, color=black] ({\i*\r+\j*\u+\k*\x+0*\w},{\i*\s+\j*\v+\k*\y+0*\z}) -- ({\i*\r+\j*\u+\k*\x+\w},{\i*\s+\j*\v+\k*\y+\z});
\foreach \i in {0,...,1}
    \foreach \j in {0,...,1}
        \foreach \k in {0,...,1}
            \foreach \h in {0,...,1}
                \draw node at ({\i*\r+\j*\u+\k*\x+\h*\w},{\i*\s+\j*\v+\k*\y+\h*\z}) {};
\draw node [label=above: {$q$}] at ({\u+\r+\x+\w},{\v+\s+\y+\z}) {};
\draw node [label=above: {$s$}] at ({\u+\x+\w},{\v+\y+\z}) {};
\draw node [label=right: {$t$}] at ({0},{0}) {};
\end{tikzpicture}
$\qquad$ 
\begin{tikzpicture}[scale=.4]
\tikzstyle{every node}=[fill=white,minimum size=1pt,inner sep=2pt]
\def \s {2.5}
\draw node [label=center: {\underline{level}}] at (0,{5*\s}) {};
\foreach \i in {0,...,4}
    \pgfmathsetmacro{\l}{int(\i+3)}
    \draw node [label=center: {$\l$}] at (0,{\i*\s}) {};
\end{tikzpicture}
$\qquad$
\begin{tikzpicture}[scale=.4]
\tikzstyle{every node}=[draw,circle,fill=black,minimum size=1pt,inner sep=2pt]
\def \r {-8}
\def \s {2.5}
\def \u {-2}
\def \v {2.5}
\def \x {1}
\def \y {2.5}
\def \w {4}
\def \z {2.5}
    \draw[line width=1.5pt, color=gren] ({0},{0}) -- ({\r},{\s});
    \draw[line width=1.5pt, color=gren] ({0},{0}) -- ({\u},{\v});
    \draw[line width=1.5pt, color=blue] ({\w},{\z}) -- ({\u+\w},{\v+\z});
    \draw[line width=1.5pt, color=blue] ({\x+\w},{\y+\z}) -- ({\u+\x+\w},{\v+\y+\z});
    \draw[line width=1.5pt, color=red] ({\r+\w},{\s+\z}) -- ({\r+\u+\w},{\s+\v+\z});
    \draw[line width=1.5pt, color=red] ({\r+\x+\w},{\s+\y+\z}) -- ({\r+\u+\x+\w},{\s+\v+\y+\z});
    \draw[line width=1.5pt, color=blue] ({\u+\w},{\v+\z}) -- ({\u+\x+\w},{\v+\y+\z});
    \draw[line width=1.5pt, color=red] ({\r+\u},{\s+\v}) -- ({\r+\u+\x},{\s+\v+\y});
    \draw[line width=1.5pt, color=red] ({\r+\u+\w},{\s+\v+\z}) -- ({\r+\u+\x+\w},{\s+\v+\y+\z});
    \draw[line width=1.5pt, color=blue] ({\x},{\y}) -- ({\x+\w},{\y+\z});
    \draw[line width=1.5pt, color=red] ({\r+\x},{\s+\y}) -- ({\r+\x+\w},{\s+\y+\z});
    \draw[line width=1.5pt, color=blue] ({\u+\x},{\v+\y}) -- ({\u+\x+\w},{\v+\y+\z});
    \draw[line width=1.5pt, color=red] ({\u+\r+\x},{\v+\s+\y}) -- ({\u+\r+\x+\w},{\v+\s+\y+\z});
\foreach \i in {0,...,1}
    \foreach \j in {0,...,1}
        \foreach \k in {0,...,1}
            \foreach \h in {0,...,1}
                \draw node at ({\i*\r+\j*\u+\k*\x+\h*\w},{\i*\s+\j*\v+\k*\y+\h*\z}) {};
\draw node [label=above: {\color{red} $q$}] at ({\u+\r+\x+\w},{\v+\s+\y+\z}) {};
\draw node [label=above: {\color{blue} $s$}] at ({\u+\x+\w},{\v+\y+\z}) {};
\draw node [label=right: {\color{gren} $t$}] at ({0},{0}) {};
\end{tikzpicture}
\caption{A level-3 4-cube $T$, left, with the $\bz$-stacking partition into spiders $\{{\color{red} (H_1,D_1)},{\color{blue} (H_2,D_2)},{\color{gren} (H_3,D_3)}\}$, right.}
\label{fig:Tesseract}
\end{center}
\end{figure}
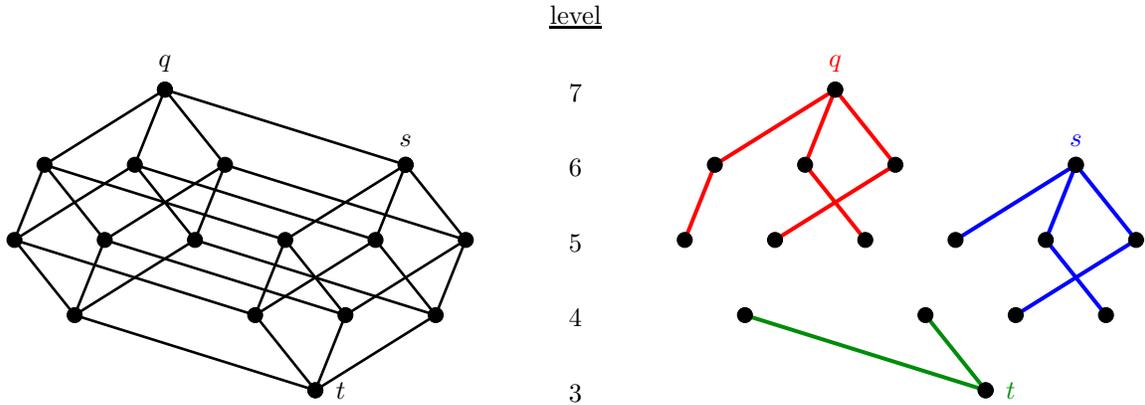

\begin{lem}
\label{l:7Cube}
For every $l\le 3$ and $d\ge l+4$, any level-$l$ 4-cube in $Q^d$ is $\bz$-stackable.
For every $d\ge 8$, the set of all level-4 3-cubes in $Q^d$ is $\bz$-stackable.
\end{lem}

\begin{figure}[ht]
\begin{center}
\begin{tikzpicture}[scale=.4]
\tikzstyle{every node}=[draw,circle,fill=black,minimum size=1pt,inner sep=2pt]
\def \r {-8}
\def \s {0}
\def \u {-3}
\def \v {2.5}
\def \x {0}
\def \y {2.5}
\def \w {3}
\def \z {2.5}
\foreach \i in {0,...,1}
    \foreach \k in {0,...,1}
        \foreach \h in {0,...,1}
            \draw[line width=1.0pt, color=black] ({\i*\r+0*\u+\k*\x+\h*\w},{\i*\s+0*\v+\k*\y+\h*\z}) -- ({\i*\r+1*\u+\k*\x+\h*\w},{\i*\s+1*\v+\k*\y+\h*\z});
\foreach \i in {0,...,1}
    \foreach \j in {0,...,1}
        \foreach \h in {0,...,1}
            \draw[line width=1.0pt, color=black] ({\i*\r+\j*\u+0*\x+\h*\w},{\i*\s+\j*\v+0*\y+\h*\z}) -- ({\i*\r+\j*\u+\x+\h*\w},{\i*\s+\j*\v+\y+\h*\z});
\foreach \i in {0,...,1}
    \foreach \j in {0,...,1}
        \foreach \k in {0,...,1}
            \draw[line width=1.0pt, color=black] ({\i*\r+\j*\u+\k*\x+0*\w},{\i*\s+\j*\v+\k*\y+0*\z}) -- ({\i*\r+\j*\u+\k*\x+\w},{\i*\s+\j*\v+\k*\y+\z});
\foreach \i in {0,...,1}
    \foreach \j in {0,...,1}
        \foreach \k in {0,...,1}
            \foreach \h in {0,...,1}
                \draw node at ({\i*\r+\j*\u+\k*\x+\h*\w},{\i*\s+\j*\v+\k*\y+\h*\z}) {};
\draw node (U) [label=left: {$u$}] at ({\u+\r+\x+\w},{\v+\s+\y+\z}) {};
\draw node (V) [label=right: {$v$}] at ({\u+\x+\w},{\v+\y+\z}) {};
\draw node (X) [label=above: {$x$}] at ({\u+\r/2+\x+\w},{\v+\y+\y+\z}) {};
\draw node [label=below: {\color{gren} $0$}] at ({\u+\r/2+\x+\w},{\v+\y+\y+\z}) {};
\draw[line width=1.0pt, color=black] (U) -- (X) -- (V);
\end{tikzpicture}
$\qquad$ 
\begin{tikzpicture}[scale=.4]
\tikzstyle{every node}=[fill=white,minimum size=1pt,inner sep=2pt]
\def \s {2.5}
\draw node [label=center: {\underline{level}}] at (0,{5*\s}) {};
\foreach \i in {0,...,4}
    \pgfmathsetmacro{\l}{int(\i+4)}
    \draw node [label=center: {$\l$}] at (0,{\i*\s}) {};
\end{tikzpicture}
$\qquad$
\begin{tikzpicture}[scale=.4]
\tikzstyle{every node}=[draw,circle,fill=black,minimum size=1pt,inner sep=2pt]
\def \r {-8}
\def \s {0}
\def \u {-3}
\def \v {2.5}
\def \x {0}
\def \y {2.5}
\def \w {3}
\def \z {2.5}
    \draw[line width=1.5pt, color=gren] ({\r},{\s}) -- ({\r+\u},{\s+\v});
    \draw[line width=1.5pt, color=gren] ({\r},{\s}) -- ({\r+\x},{\s+\y});
    \draw[line width=1.5pt, color=gren] ({\r+\u},{\s+\v}) -- ({\r+\u+\x},{\s+\v+\y});
    \draw[line width=1.5pt, color=gren] ({\r+\u},{\s+\v}) -- ({\r+\u+\w},{\s+\v+\z});
    \draw[line width=1.5pt, color=red] ({\r+\w},{\s+\z}) -- ({\r+\x+\w},{\s+\y+\z});
    \draw[line width=1.5pt, color=red] ({\r+\x+\w},{\s+\y+\z}) -- ({\r+\u+\x+\w},{\s+\v+\y+\z});
    \draw[line width=1.5pt, color=red] ({\w},{\z}) -- ({\x+\w},{\y+\z});
    \draw[line width=1.5pt, color=red] ({\x+\w},{\y+\z}) -- ({\u+\x+\w},{\v+\y+\z});
    \draw[line width=1.5pt, color=blue] ({0},{0}) -- ({\u},{\v});
    \draw[line width=1.5pt, color=blue] ({0},{0}) -- ({\x},{\y});
    \draw[line width=1.5pt, color=blue] ({\u},{\v}) -- ({\u+\x},{\v+\y});
    \draw[line width=1.5pt, color=blue] ({\u},{\v}) -- ({\u+\w},{\v+\z});
\foreach \i in {0,...,1}
    \foreach \j in {0,...,1}
        \foreach \k in {0,...,1}
            \foreach \h in {0,...,1}
                \draw node at ({\i*\r+\j*\u+\k*\x+\h*\w},{\i*\s+\j*\v+\k*\y+\h*\z}) {};
\draw node [label=left: {\color{gren} $a$}] at ({\r+\u},{\s+\v}) {};
\draw node [label=above: {\color{red} $b$}] at ({\r+\x+\w},{\s+\y+\z}) {};
\draw node [label=below: {\color{blue} $c$}] at ({\u},{\v}) {};
\draw node (U) at ({\u+\r+\x+\w},{\v+\s+\y+\z}) {};
\draw node (V) [label=right: {$v$}] at ({\u+\x+\w},{\v+\y+\z}) {};
\draw node (X) [label=above: {$x$}] at ({\u+\r/2+\x+\w},{\v+\y+\y+\z}) {};
\draw[line width=1.5pt, color=red] (U) -- (X) -- (V);
\end{tikzpicture}
\caption{Coordinating the stacking of two ``adjacent'' level-4 3-cubes in $Q^d$ for $d\ge 8$.}
\label{fig:11CubesP2}
\end{center}
\end{figure}
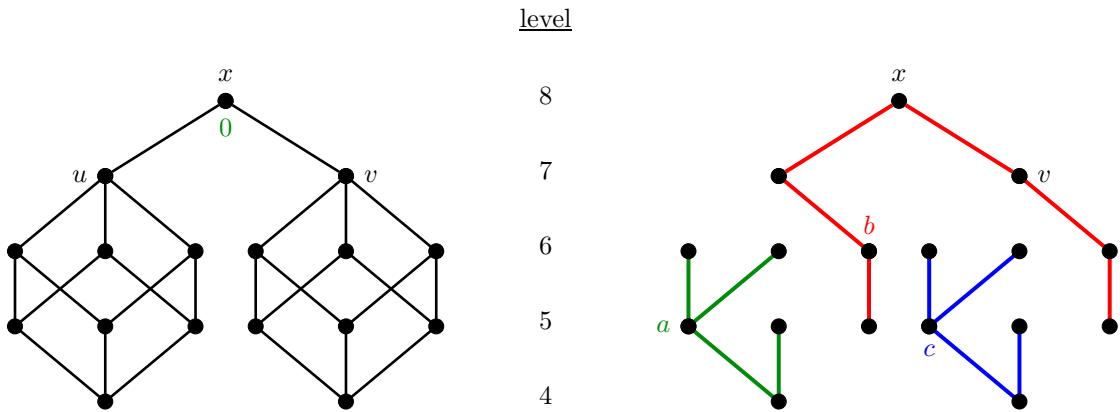

\begin{proof}
Let $A$ be a level-$l$ 4-cube.
If $l\le 2$ then it can be partitioned into two 3-cubes of levels at most 3, so the result follows from Lemma \ref{l:LevelkCubes}.
If $l=3$ then Figure \ref{fig:Tesseract} shows a $\bz$-stacking partition.
If $l=4$ then we partition it into two 3-cubes of levels 4 and 5.
We stack each level-5 3-cube to its top vertex, leaving the set of all level-4 3-cubes in $Q^d$ to stack.
These we now consider.

These cubes can be labeled by their lowest vertex; the set of such labels is $\binom{[d-3]}{4} = V(J(d-3,4,3))$.
By Theorem \ref{t:JohnsonHam} we let $H$ be a Hamilton cycle in $J(d-3,4,3)$, which we partition into paths of length 1, along with a single path of length 2 if $\binom{d-3}{4}$ is odd.
For each path $P$ in this partition we show below how the union of cubes associated with the vertices of $P$ is $r$-stackable.

When $P$ consists of two adjacent vertices $U$ and $V$ we set $X=U\cup V$; recall that $|U|=|V|=4$ and $|U\cap V|=3$, and so $|X|=5$; i.e. $X$ corresponds to a level-5 3-cube.
Hence the top vertices $u$ and $v$ of the cubes labeled $U$ and $V$ are each adjacent in $Q^d$ to the top vertex $x$ of the cube labeled $X$ (see the left side of Figure \ref{fig:11CubesP2}).
In this case, the configuration of cups on this structure $S$ has one cup on every vertex except no cup on vertex $x$, since all the cups in the 3-cube labeled $X$ were already stacked onto its top vertex.
The right side of Figure \ref{fig:11CubesP2} shows in color the stacking partition of $S$ that places the appropriate number of cups on the appropriate levels of vertices $a$, $b$, and $c$ (which are roots of spiders), to then stack them all on $r$.
The point of vertex $x$ is for it to be used to move three cups from $v$ to $b$.

\begin{figure}[ht]
\begin{center}
\begin{tikzpicture}[scale=.4]
\tikzstyle{every node}=[draw,circle,fill=black,minimum size=1pt,inner sep=2pt]
\def \r {-8}
\def \R {8}
\def \s {0}
\def \u {-3}
\def \v {2.5}
\def \x {0}
\def \y {2.5}
\def \w {3}
\def \z {2.5}
\foreach \i in {0,...,1}
    \foreach \k in {0,...,1}
        \foreach \h in {0,...,1}
            \draw[line width=0.5pt, color=black] ({\i*\r+0*\u+\k*\x+\h*\w},{\i*\s+0*\v+\k*\y+\h*\z}) -- ({\i*\r+1*\u+\k*\x+\h*\w},{\i*\s+1*\v+\k*\y+\h*\z});
\foreach \i in {0,...,1}
    \foreach \k in {0,...,1}
        \foreach \h in {0,...,1}
            \draw[line width=0.5pt, color=black] ({\i*\R+0*\u+\k*\x+\h*\w},{\i*\s+0*\v+\k*\y+\h*\z}) -- ({\i*\R+1*\u+\k*\x+\h*\w},{\i*\s+1*\v+\k*\y+\h*\z});
\foreach \i in {0,...,1}
    \foreach \j in {0,...,1}
        \foreach \h in {0,...,1}
            \draw[line width=0.5pt, color=black] ({\i*\r+\j*\u+0*\x+\h*\w},{\i*\s+\j*\v+0*\y+\h*\z}) -- ({\i*\r+\j*\u+\x+\h*\w},{\i*\s+\j*\v+\y+\h*\z});
\foreach \i in {0,...,1}
    \foreach \j in {0,...,1}
        \foreach \h in {0,...,1}
            \draw[line width=0.5pt, color=black] ({\i*\R+\j*\u+0*\x+\h*\w},{\i*\s+\j*\v+0*\y+\h*\z}) -- ({\i*\R+\j*\u+\x+\h*\w},{\i*\s+\j*\v+\y+\h*\z});
\foreach \i in {0,...,1}
    \foreach \j in {0,...,1}
        \foreach \k in {0,...,1}
            \draw[line width=0.5pt, color=black] ({\i*\r+\j*\u+\k*\x+0*\w},{\i*\s+\j*\v+\k*\y+0*\z}) -- ({\i*\r+\j*\u+\k*\x+\w},{\i*\s+\j*\v+\k*\y+\z});
\foreach \i in {0,...,1}
    \foreach \j in {0,...,1}
        \foreach \k in {0,...,1}
            \draw[line width=0.5pt, color=black] ({\i*\R+\j*\u+\k*\x+0*\w},{\i*\s+\j*\v+\k*\y+0*\z}) -- ({\i*\R+\j*\u+\k*\x+\w},{\i*\s+\j*\v+\k*\y+\z});
\foreach \i in {0,...,1}
    \foreach \j in {0,...,1}
        \foreach \k in {0,...,1}
            \foreach \h in {0,...,1}
                \draw node at ({\i*\r+\j*\u+\k*\x+\h*\w},{\i*\s+\j*\v+\k*\y+\h*\z}) {};
\foreach \i in {0,...,1}
    \foreach \j in {0,...,1}
        \foreach \k in {0,...,1}
            \foreach \h in {0,...,1}
                \draw node at ({\i*\R+\j*\u+\k*\x+\h*\w},{\i*\s+\j*\v+\k*\y+\h*\z}) {};
\draw node (U) [label=left: {$u$}] at ({\u+\r+\x+\w},{\v+\s+\y+\z}) {};
\draw node (V) [label=above: {$v$}] at ({\u+\x+\w},{\v+\y+\z}) {};
\draw node (W) [label=right: {$w$}] at ({\u+\R+\x+\w},{\v+\s+\y+\z}) {};
\draw node (X) [label=above: {$x$}] at ({\u+\r/2+\x+\w},{\v+\y+\y+\z}) {};
\draw node (Y) [label=above: {$y$}] at ({\u+\R/2+\x+\w},{\v+\y+\y+\z}) {};
\draw node [label=below: {\color{gren} $0$}] at ({\u+\r/2+\x+\w},{\v+\y+\y+\z}) {};
\draw node [label=below: {\color{gren} $0$}] at ({\u+\R/2+\x+\w},{\v+\y+\y+\z}) {};
\draw[line width=0.5pt, color=black] (U) -- (X) -- (V);
\draw[line width=0.5pt, color=black] (V) -- (Y) -- (W);
    \draw[line width=1.5pt, color=red] ({\r+\x},{\s+\y}) -- ({\r+\u+\x},{\s+\v+\y});
    \draw[line width=1.5pt, color=red] ({\r},{\s}) -- ({\r+\x},{\s+\y});
    \draw[line width=1.5pt, color=red] ({\r+\u},{\s+\v}) -- ({\r+\u+\x},{\s+\v+\y});
    \draw[line width=1.5pt, color=red] ({\r+\u},{\s+\v}) -- ({\r+\u+\w},{\s+\v+\z});
    \draw[line width=1.5pt, color=red] ({\r+\w},{\s+\z}) -- ({\r+\u+\w},{\s+\v+\z});
    \draw[line width=1.5pt, color=gren] ({\r+\x+\w},{\s+\y+\z}) -- ({\r+\u+\x+\w},{\s+\v+\y+\z});
    \draw[line width=1.5pt, color=gren] ({\R+\x+\w},{\s+\y+\z}) -- ({\R+\u+\x+\w},{\s+\v+\y+\z});
    \draw[line width=1.5pt, color=gren] ({\w},{\z}) -- ({\x+\w},{\y+\z});
    \draw[line width=1.5pt, color=gren] ({\x+\w},{\y+\z}) -- ({\u+\x+\w},{\v+\y+\z});
    \draw[line width=1.5pt, color=blue] ({0},{0}) -- ({\u},{\v});
    \draw[line width=1.5pt, color=blue] ({0},{0}) -- ({\x},{\y});
    \draw[line width=1.5pt, color=blue] ({\u},{\v}) -- ({\u+\x},{\v+\y});
    \draw[line width=1.5pt, color=blue] ({\u},{\v}) -- ({\u+\w},{\v+\z});
    \draw[line width=1.5pt, color=orange] ({\R+\x},{\s+\y}) -- ({\R+\u+\x},{\s+\v+\y});
    \draw[line width=1.5pt, color=orange] ({\R},{\s}) -- ({\R+\x},{\s+\y});
    \draw[line width=1.5pt, color=orange] ({\R+\u},{\s+\v}) -- ({\R+\u+\x},{\s+\v+\y});
    \draw[line width=1.5pt, color=orange] ({\R+\u},{\s+\v}) -- ({\R+\u+\w},{\s+\v+\z});
    \draw[line width=1.5pt, color=orange] ({\R+\w},{\s+\z}) -- ({\R+\u+\w},{\s+\v+\z});
\foreach \i in {0,...,1}
    \foreach \j in {0,...,1}
        \foreach \k in {0,...,1}
            \foreach \h in {0,...,1}
                \draw node at ({\i*\r+\j*\u+\k*\x+\h*\w},{\i*\s+\j*\v+\k*\y+\h*\z}) {};
\foreach \i in {0,...,1}
    \foreach \j in {0,...,1}
        \foreach \k in {0,...,1}
            \foreach \h in {0,...,1}
                \draw node at ({\i*\R+\j*\u+\k*\x+\h*\w},{\i*\s+\j*\v+\k*\y+\h*\z}) {};
\draw node [label=left: {\color{red} $a$}] at ({\r+\u+\x},{\s+\v+\y}) {};
\draw node [label=below: {\color{blue} $b$}] at ({\u},{\v}) {};
\draw node [label=above: {\color{orange} $c$}] at ({\R+\u+\x},{\s+\v+\y}) {};
\draw node (U) [label=left: {$u$}] at ({\u+\r+\x+\w},{\v+\s+\y+\z}) {};
\draw node (V) [label=above: {\color{gren} $v$}] at ({\u+\x+\w},{\v+\y+\z}) {};
\draw node (W) [label=right: {$w$}] at ({\u+\R+\x+\w},{\v+\s+\y+\z}) {};
\draw node (X) [label=above: {$x$}] at ({\u+\r/2+\x+\w},{\v+\y+\y+\z}) {};
\draw node (Y) [label=above: {$y$}] at ({\u+\R/2+\x+\w},{\v+\y+\y+\z}) {};
\draw node [label=below: {\color{gren} $0$}] at ({\u+\r/2+\x+\w},{\v+\y+\y+\z}) {};
\draw node [label=below: {\color{gren} $0$}] at ({\u+\R/2+\x+\w},{\v+\y+\y+\z}) {};
\draw[line width=1.5pt, color=gren] (U) -- (X) -- (V);
\draw[line width=1.5pt, color=gren] (V) -- (Y) -- (W);
\end{tikzpicture}
$\qquad$ 
\begin{tikzpicture}[scale=.4]
\tikzstyle{every node}=[fill=white,minimum size=1pt,inner sep=2pt]
\def \s {2.5}
\draw node [label=center: {\underline{level}}] at (0,{5*\s}) {};
\foreach \i in {0,...,4}
    \pgfmathsetmacro{\l}{int(\i+4)}
    \draw node [label=center: {$\l$}] at (0,{\i*\s}) {};
\end{tikzpicture}
\caption{Coordinating the stacking of three ``consecutive'' level-4 3-cubes in $Q^d$ for $d\ge 8$.}
\label{fig:11CubesP3}
\end{center}
\end{figure}
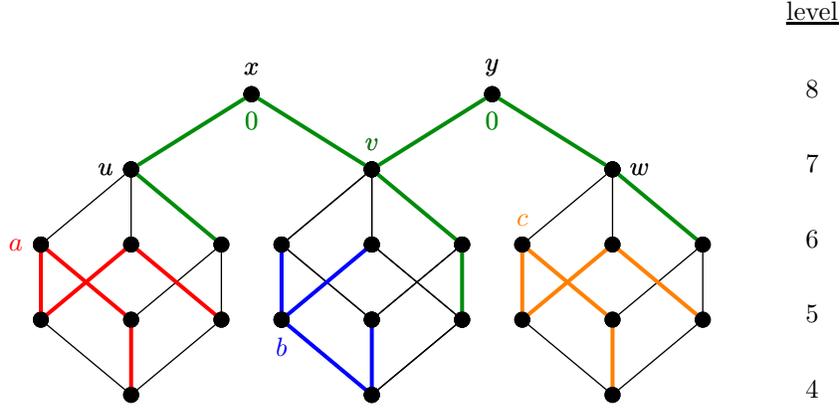

When $P$ consists of three consecutively adjacent vertices $U$, $V$, and $W$, we set $X=U\cup V$ and $Y=V\cup W$; so that $|X|=|Y|=4$.
Hence the top vertices $u$ and $v$ of the cubes labeled $U$ and $V$ are each adjacent in $Q^d$ to the top vertex $x$ of the cube labeled $X$, while the top vertices $v$ and $w$ of the cubes labeled $V$ and $W$ are each adjacent in $Q^d$ to the top vertex $y$ of the cube labeled $Y$ (see Figure \ref{fig:11CubesP3}).
In this case, the configuration of cups on this structure $S$ has one cup on every vertex except no cups on vertices $x$ and $y$, since all the cups in the 3-cubes labeled $X$ and $Y$ were already stacked onto the top vertices of those cubes.
Figure \ref{fig:11CubesP2} also shows in color the stacking partition of $S$ that places the appropriate number of cups on the appropriate levels of vertices $a$, $b$, $c$, and $v$ (roots of spiders) to then stack them all on $r$.
The point of vertices $x$ and $y$ is for them to be used to move two cups  from each of $u$ and $w$ to $v$.

This completes the proof.
\end{proof}

\begin{cor}
\label{c:11Cube}
For $d\le 11$, $Q^d$ is stackable.
\end{cor}

\begin{proof}
The cases $d\le 6$ is proven in Lemma \ref{l:LevelkCubes}.

When $d=7$ we write $Q^7=Q^3\ssq Q^4$, which partitions $Q^d$ into 4-cubes having levels at most 3, which are $\bz$-stackable by Lemma \ref{l:7Cube}.

When $8\le d\le 11$ we write $Q^d=Q^{d-3}\ssq Q^3$, which partitions $Q^d$ into 3-cubes having levels at most 8, which are $\bz$-stackable by Lemmas \ref{l:LevelkCubes} and \ref{l:7Cube}.
\end{proof}

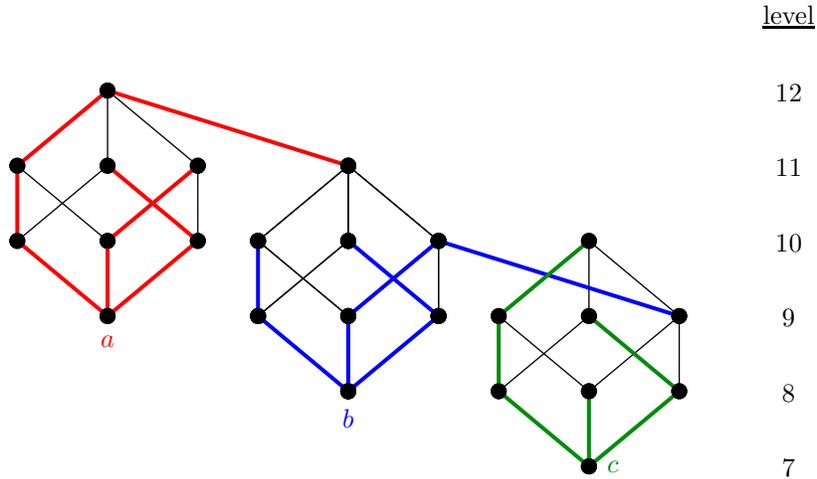
\begin{figure}[ht]
\begin{center}
\begin{tikzpicture}[scale=.4]
\tikzstyle{every node}=[draw,circle,fill=black,minimum size=1pt,inner sep=2pt]
\def \r {-8}
\def \R {8}
\def \s {2.5}
\def \S {-2.5}
\def \u {-3}
\def \v {2.5}
\def \x {0}
\def \y {2.5}
\def \w {3}
\def \z {2.5}
\foreach \i in {0,...,1}
    \foreach \k in {0,...,1}
        \foreach \h in {0,...,1}
            \draw[line width=0.5pt, color=black] ({\i*\r+0*\u+\k*\x+\h*\w},{\i*\s+0*\v+\k*\y+\h*\z}) -- ({\i*\r+1*\u+\k*\x+\h*\w},{\i*\s+1*\v+\k*\y+\h*\z});
\foreach \i in {0,...,1}
    \foreach \k in {0,...,1}
        \foreach \h in {0,...,1}
            \draw[line width=0.5pt, color=black] ({\i*\R+0*\u+\k*\x+\h*\w},{\i*\S+0*\v+\k*\y+\h*\z}) -- ({\i*\R+1*\u+\k*\x+\h*\w},{\i*\S+1*\v+\k*\y+\h*\z});
\foreach \i in {0,...,1}
    \foreach \j in {0,...,1}
        \foreach \h in {0,...,1}
            \draw[line width=0.5pt, color=black] ({\i*\r+\j*\u+0*\x+\h*\w},{\i*\s+\j*\v+0*\y+\h*\z}) -- ({\i*\r+\j*\u+\x+\h*\w},{\i*\s+\j*\v+\y+\h*\z});
\foreach \i in {0,...,1}
    \foreach \j in {0,...,1}
        \foreach \h in {0,...,1}
            \draw[line width=0.5pt, color=black] ({\i*\R+\j*\u+0*\x+\h*\w},{\i*\S+\j*\v+0*\y+\h*\z}) -- ({\i*\R+\j*\u+\x+\h*\w},{\i*\S+\j*\v+\y+\h*\z});
\foreach \i in {0,...,1}
    \foreach \j in {0,...,1}
        \foreach \k in {0,...,1}
            \draw[line width=0.5pt, color=black] ({\i*\r+\j*\u+\k*\x+0*\w},{\i*\s+\j*\v+\k*\y+0*\z}) -- ({\i*\r+\j*\u+\k*\x+\w},{\i*\s+\j*\v+\k*\y+\z});
\foreach \i in {0,...,1}
    \foreach \j in {0,...,1}
        \foreach \k in {0,...,1}
            \draw[line width=0.5pt, color=black] ({\i*\R+\j*\u+\k*\x+0*\w},{\i*\S+\j*\v+\k*\y+0*\z}) -- ({\i*\R+\j*\u+\k*\x+\w},{\i*\S+\j*\v+\k*\y+\z});
\foreach \i in {0,...,1}
    \foreach \j in {0,...,1}
        \foreach \k in {0,...,1}
            \foreach \h in {0,...,1}
                \draw node at ({\i*\r+\j*\u+\k*\x+\h*\w},{\i*\s+\j*\v+\k*\y+\h*\z}) {};
\foreach \i in {0,...,1}
    \foreach \j in {0,...,1}
        \foreach \k in {0,...,1}
            \foreach \h in {0,...,1}
                \draw node at ({\i*\R+\j*\u+\k*\x+\h*\w},{\i*\S+\j*\v+\k*\y+\h*\z}) {};
    \draw[line width=1.5pt, color=red] ({\r},{\s}) -- ({\r+\u},{\s+\v});
    \draw[line width=1.5pt, color=red] ({\r},{\s}) -- ({\r+\x},{\s+\y});
    \draw[line width=1.5pt, color=red] ({\r},{\s}) -- ({\r+\w},{\s+\z});
    \draw[line width=1.5pt, color=red] ({\r+\u},{\s+\v}) -- ({\r+\u+\x},{\s+\v+\y});
    \draw[line width=1.5pt, color=red] ({\r+\x},{\s+\y}) -- ({\r+\x+\w},{\s+\y+\z});
    \draw[line width=1.5pt, color=red] ({\r+\w},{\s+\z}) -- ({\r+\u+\w},{\s+\v+\z});
    \draw[line width=1.5pt, color=red] ({\r+\u+\x},{\s+\v+\y}) -- ({\r+\u+\x+\w},{\s+\v+\y+\z});
    \draw[line width=1.5pt, color=red] ({\u+\x+\w},{\v+\y+\z}) -- ({\r+\u+\x+\w},{\s+\v+\y+\z});
    \draw[line width=1.5pt, color=blue] ({0},{0}) -- ({\u},{\v});
    \draw[line width=1.5pt, color=blue] ({\w},{\z}) -- ({\u+\w},{\v+\z});
    \draw[line width=1.5pt, color=blue] ({0},{0}) -- ({\x},{\y});
    \draw[line width=1.5pt, color=blue] ({0},{0}) -- ({\w},{\z});
    \draw[line width=1.5pt, color=blue] ({\u},{\v}) -- ({\u+\x},{\v+\y});
    \draw[line width=1.5pt, color=blue] ({\x},{\y}) -- ({\x+\w},{\y+\z});
    \draw[line width=1.5pt, color=blue] ({\x+\w},{\y+\z}) -- ({\R+\x+\w},{\S+\y+\z});
    \draw[line width=1.5pt, color=gren] ({\R},{\S}) -- ({\R+\u},{\S+\v});
    \draw[line width=1.5pt, color=gren] ({\R},{\S}) -- ({\R+\x},{\S+\y});
    \draw[line width=1.5pt, color=gren] ({\R},{\S}) -- ({\R+\w},{\S+\z});
    \draw[line width=1.5pt, color=gren] ({\R+\u},{\S+\v}) -- ({\R+\u+\x},{\S+\v+\y});
    \draw[line width=1.5pt, color=gren] ({\R+\u+\x},{\S+\v+\y}) -- ({\R+\u+\x+\w},{\S+\v+\y+\z});
    \draw[line width=1.5pt, color=gren] ({\R+\w},{\S+\z}) -- ({\R+\u+\w},{\S+\v+\z});
\foreach \i in {0,...,1}
    \foreach \j in {0,...,1}
        \foreach \k in {0,...,1}
            \foreach \h in {0,...,1}
                \draw node at ({\i*\r+\j*\u+\k*\x+\h*\w},{\i*\s+\j*\v+\k*\y+\h*\z}) {};
\foreach \i in {0,...,1}
    \foreach \j in {0,...,1}
        \foreach \k in {0,...,1}
            \foreach \h in {0,...,1}
                \draw node at ({\i*\R+\j*\u+\k*\x+\h*\w},{\i*\S+\j*\v+\k*\y+\h*\z}) {};
\draw node [label=below: {\color{red} $a$}] at ({\r},{\s}) {};
\draw node [label=below: {\color{blue} $b$}] at ({0},{0}) {};
\draw node [label=right: {\color{gren} $c$}] at ({\R},{\S}) {};
\end{tikzpicture}
$\qquad$ 
\begin{tikzpicture}[scale=.4]
\tikzstyle{every node}=[fill=white,minimum size=1pt,inner sep=2pt]
\def \s {2.5}
\draw node [label=center: {\underline{level}}] at (0,{6*\s}) {};
\foreach \i in {0,...,5}
    \pgfmathsetmacro{\l}{int(\i+7)}
    \draw node [label=center: {$\l$}] at (0,{\i*\s}) {};
\end{tikzpicture}
\caption{Coordinating the $\bz$-stacking of 3-cubes at levels 7--9 in $Q^{12}$.}
\label{fig:12Cubes}
\end{center}
\end{figure}

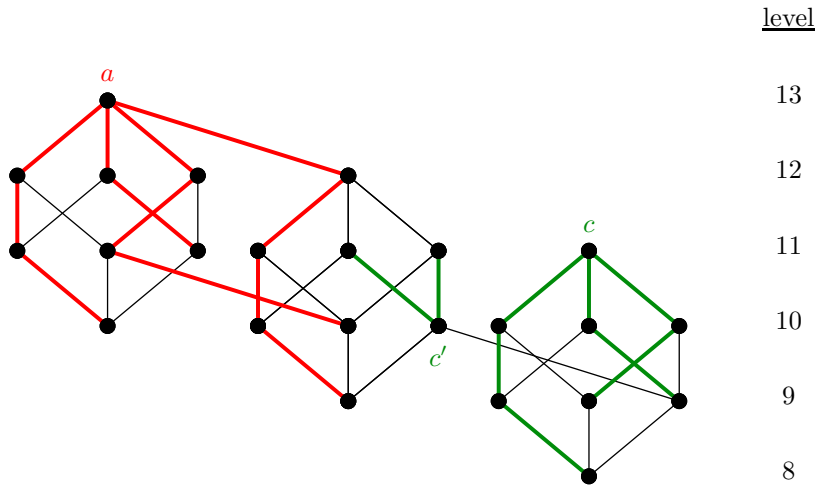
\begin{figure}[ht]
\begin{center}
\begin{tikzpicture}[scale=.4]
\tikzstyle{every node}=[draw,circle,fill=black,minimum size=1pt,inner sep=2pt]
\def \r {-8}
\def \R {8}
\def \s {2.5}
\def \S {-2.5}
\def \u {-3}
\def \v {2.5}
\def \x {0}
\def \y {2.5}
\def \w {3}
\def \z {2.5}
\foreach \i in {0,...,1}
    \foreach \k in {0,...,1}
        \foreach \h in {0,...,1}
            \draw[line width=0.5pt, color=black] ({\i*\r+0*\u+\k*\x+\h*\w},{\i*\s+0*\v+\k*\y+\h*\z}) -- ({\i*\r+1*\u+\k*\x+\h*\w},{\i*\s+1*\v+\k*\y+\h*\z});
\foreach \i in {0,...,1}
    \foreach \k in {0,...,1}
        \foreach \h in {0,...,1}
            \draw[line width=0.5pt, color=black] ({\i*\R+0*\u+\k*\x+\h*\w},{\i*\S+0*\v+\k*\y+\h*\z}) -- ({\i*\R+1*\u+\k*\x+\h*\w},{\i*\S+1*\v+\k*\y+\h*\z});
\foreach \i in {0,...,1}
    \foreach \j in {0,...,1}
        \foreach \h in {0,...,1}
            \draw[line width=0.5pt, color=black] ({\i*\r+\j*\u+0*\x+\h*\w},{\i*\s+\j*\v+0*\y+\h*\z}) -- ({\i*\r+\j*\u+\x+\h*\w},{\i*\s+\j*\v+\y+\h*\z});
\foreach \i in {0,...,1}
    \foreach \j in {0,...,1}
        \foreach \h in {0,...,1}
            \draw[line width=0.5pt, color=black] ({\i*\R+\j*\u+0*\x+\h*\w},{\i*\S+\j*\v+0*\y+\h*\z}) -- ({\i*\R+\j*\u+\x+\h*\w},{\i*\S+\j*\v+\y+\h*\z});
\foreach \i in {0,...,1}
    \foreach \j in {0,...,1}
        \foreach \k in {0,...,1}
            \draw[line width=0.5pt, color=black] ({\i*\r+\j*\u+\k*\x+0*\w},{\i*\s+\j*\v+\k*\y+0*\z}) -- ({\i*\r+\j*\u+\k*\x+\w},{\i*\s+\j*\v+\k*\y+\z});
\foreach \i in {0,...,1}
    \foreach \j in {0,...,1}
        \foreach \k in {0,...,1}
            \draw[line width=0.5pt, color=black] ({\i*\R+\j*\u+\k*\x+0*\w},{\i*\S+\j*\v+\k*\y+0*\z}) -- ({\i*\R+\j*\u+\k*\x+\w},{\i*\S+\j*\v+\k*\y+\z});
\foreach \i in {0,...,1}
    \foreach \j in {0,...,1}
        \foreach \k in {0,...,1}
            \foreach \h in {0,...,1}
                \draw node at ({\i*\r+\j*\u+\k*\x+\h*\w},{\i*\s+\j*\v+\k*\y+\h*\z}) {};
\foreach \i in {0,...,1}
    \foreach \j in {0,...,1}
        \foreach \k in {0,...,1}
            \foreach \h in {0,...,1}
                \draw node at ({\i*\R+\j*\u+\k*\x+\h*\w},{\i*\S+\j*\v+\k*\y+\h*\z}) {};
    \draw[line width=1.5pt, color=red] ({\r},{\s}) -- ({\r+\u},{\s+\v});
    \draw[line width=1.5pt, color=red] ({\r+\u},{\s+\v}) -- ({\r+\u+\x},{\s+\v+\y});
    \draw[line width=1.5pt, color=red] ({\r+\x},{\s+\y}) -- ({\r+\x+\w},{\s+\y+\z});
    \draw[line width=1.5pt, color=red] ({\r+\w},{\s+\z}) -- ({\r+\u+\w},{\s+\v+\z});
    \draw[line width=1.5pt, color=red] ({\r+\u+\x},{\s+\v+\y}) -- ({\r+\u+\x+\w},{\s+\v+\y+\z});
    \draw[line width=1.5pt, color=red] ({\r+\u+\w},{\s+\v+\z}) -- ({\r+\u+\x+\w},{\s+\v+\y+\z});
    \draw[line width=1.5pt, color=red] ({\r+\x+\w},{\s+\y+\z}) -- ({\r+\u+\x+\w},{\s+\v+\y+\z});
    \draw[line width=1.5pt, color=red] ({\x},{\y}) -- ({\r+\x},{\s+\y});
    \draw[line width=1.5pt, color=red] ({\u+\x+\w},{\v+\y+\z}) -- ({\r+\u+\x+\w},{\s+\v+\y+\z});
    \draw[line width=1.5pt, color=red] ({0},{0}) -- ({\u},{\v});
    \draw[line width=1.5pt, color=red] ({\u},{\v}) -- ({\u+\x},{\v+\y});
    \draw[line width=1.5pt, color=gren] ({\w},{\z}) -- ({\u+\w},{\v+\z});
    \draw[line width=1.5pt, color=gren] ({\w},{\z}) -- ({\x+\w},{\y+\z});
    \draw[line width=1.5pt, color=red] ({\u+\x},{\v+\y}) -- ({\u+\x+\w},{\v+\y+\z});
    \draw[line width=0.5pt] ({\w},{\z}) -- ({\R+\w},{\S+\z});
    \draw[line width=1.5pt, color=gren] ({\R},{\S}) -- ({\R+\u},{\S+\v});
    \draw[line width=1.5pt, color=gren] ({\R+\u},{\S+\v}) -- ({\R+\u+\x},{\S+\v+\y});
    \draw[line width=1.5pt, color=gren] ({\R+\x},{\S+\y}) -- ({\R+\x+\w},{\S+\y+\z});
    \draw[line width=1.5pt, color=gren] ({\R+\u+\x},{\S+\v+\y}) -- ({\R+\u+\x+\w},{\S+\v+\y+\z});
    \draw[line width=1.5pt, color=gren] ({\R+\u+\w},{\S+\v+\z}) -- ({\R+\u+\x+\w},{\S+\v+\y+\z});
    \draw[line width=1.5pt, color=gren] ({\R+\x+\w},{\S+\y+\z}) -- ({\R+\u+\x+\w},{\S+\v+\y+\z});
    \draw[line width=1.5pt, color=gren] ({\R+\w},{\S+\z}) -- ({\R+\u+\w},{\S+\v+\z});
\foreach \i in {0,...,1}
    \foreach \j in {0,...,1}
        \foreach \k in {0,...,1}
            \foreach \h in {0,...,1}
                \draw node at ({\i*\r+\j*\u+\k*\x+\h*\w},{\i*\s+\j*\v+\k*\y+\h*\z}) {};
\foreach \i in {0,...,1}
    \foreach \j in {0,...,1}
        \foreach \k in {0,...,1}
            \foreach \h in {0,...,1}
                \draw node at ({\i*\R+\j*\u+\k*\x+\h*\w},{\i*\S+\j*\v+\k*\y+\h*\z}) {};
\draw node [label=above: {\color{red} $a$}] at ({\r+\u+\x+\w},{\s+\v+\y+\z}) {};
\draw node [label=below: {\color{gren} $c'$}] at ({\w},{\z}) {};
\draw node [label=above: {\color{gren} $c$}] at ({\R+\u+\x+\w},{\S+\v+\y+\z}) {};
\end{tikzpicture}
$\qquad$ 
\begin{tikzpicture}[scale=.4]
\tikzstyle{every node}=[fill=white,minimum size=1pt,inner sep=2pt]
\def \s {2.5}
\draw node [label=center: {\underline{level}}] at (0,{6*\s}) {};
\foreach \i in {0,...,5}
    \pgfmathsetmacro{\l}{int(\i+8)}
    \draw node [label=center: {$\l$}] at (0,{\i*\s}) {};
\end{tikzpicture}
\caption{Coordinating the $\bz$-stacking of 3-cubes at levels 8--10 in $Q^{13}$.}
\label{fig:13Cubes}
\end{center}
\end{figure}

\begin{lem}
\label{l:ABC3Cubes}
Suppose that $9\le l\le 12$ and $l+3\le d\le 2l$ and choose some injection $\f$ from Corollary \ref{c:Inject}, with $n=d-3$.
Let $A\in\binom{[n]}{l}$ be the label of a level-$l$ 3-cube in $Q^d$, $B=\f(A)$, and $C=\f(B)$.
Then $A\cup B\cup C$ is $\bz$-stackable in $Q^d$.
\end{lem}

\begin{proof}
We note first that the existence of $\f$ is given by the hypothesis of Corollary \ref{c:Inject}: if $\lev(A)=l$ then $\lev(B)=l-1> (d-3)/2$.

The stacking partitions are given in Figures \ref{fig:12Cubes}, \ref{fig:13Cubes}, \ref{fig:14Cubes}, and \ref{fig:15Cubes}.

In Figure \ref{fig:12Cubes}, a level-9 3-cube $A$ steals a vertex from the level-8 3-cube $B$ to form a spider rooted at $a$, with the resulting 9 cups stacked at $a$ before being sent to $\bz$.
Then the ``damaged'' cube $B$ steals a vertex from the level-7 3-cube $C$ to form a spider rooted at $b$, with the resulting 8 cups stacked at $b$ before being sent to $\bz$.
Finally, the damaged $C$ stacks its 7 cups at the root of the spider at $c$ before sending them to $\bz$.

In Figure \ref{fig:13Cubes}, $A$ steals 5 vertices from $B$ to form a spider at $a$, with the resulting 13 cups stacked at $a$ before being sent to $\bz$. 
Then $C$ steals the remaining three vertices from $B$ to form spiders at $c$ and $c'$, with the 3 cups stacked at $c'$ being sent to $c$, joining with the 8 cups stacked at $c$ to send the resulting 11 cups to $\bz$.

In Figures \ref{fig:14Cubes} and \ref{fig:15Cubes}, similar schemes are shown, combining stealing, spiders, and supplemental spiders, the details of which are left to the reader.
\end{proof}

\begin{figure}[ht]
\begin{center}
\begin{tikzpicture}[scale=.4]
\tikzstyle{every node}=[draw,circle,fill=black,minimum size=1pt,inner sep=2pt]
\def \r {-8}
\def \R {8}
\def \s {2.5}
\def \S {-2.5}
\def \u {-3}
\def \v {2.5}
\def \x {0}
\def \y {2.5}
\def \w {3}
\def \z {2.5}
\foreach \i in {0,...,1}
    \foreach \k in {0,...,1}
        \foreach \h in {0,...,1}
            \draw[line width=0.5pt, color=black] ({\i*\r+0*\u+\k*\x+\h*\w},{\i*\s+0*\v+\k*\y+\h*\z}) -- ({\i*\r+1*\u+\k*\x+\h*\w},{\i*\s+1*\v+\k*\y+\h*\z});
\foreach \i in {0,...,1}
    \foreach \k in {0,...,1}
        \foreach \h in {0,...,1}
            \draw[line width=0.5pt, color=black] ({\i*\R+0*\u+\k*\x+\h*\w},{\i*\S+0*\v+\k*\y+\h*\z}) -- ({\i*\R+1*\u+\k*\x+\h*\w},{\i*\S+1*\v+\k*\y+\h*\z});
\foreach \i in {0,...,1}
    \foreach \j in {0,...,1}
        \foreach \h in {0,...,1}
            \draw[line width=0.5pt, color=black] ({\i*\r+\j*\u+0*\x+\h*\w},{\i*\s+\j*\v+0*\y+\h*\z}) -- ({\i*\r+\j*\u+\x+\h*\w},{\i*\s+\j*\v+\y+\h*\z});
\foreach \i in {0,...,1}
    \foreach \j in {0,...,1}
        \foreach \h in {0,...,1}
            \draw[line width=0.5pt, color=black] ({\i*\R+\j*\u+0*\x+\h*\w},{\i*\S+\j*\v+0*\y+\h*\z}) -- ({\i*\R+\j*\u+\x+\h*\w},{\i*\S+\j*\v+\y+\h*\z});
\foreach \i in {0,...,1}
    \foreach \j in {0,...,1}
        \foreach \k in {0,...,1}
            \draw[line width=0.5pt, color=black] ({\i*\r+\j*\u+\k*\x+0*\w},{\i*\s+\j*\v+\k*\y+0*\z}) -- ({\i*\r+\j*\u+\k*\x+\w},{\i*\s+\j*\v+\k*\y+\z});
\foreach \i in {0,...,1}
    \foreach \j in {0,...,1}
        \foreach \k in {0,...,1}
            \draw[line width=0.5pt, color=black] ({\i*\R+\j*\u+\k*\x+0*\w},{\i*\S+\j*\v+\k*\y+0*\z}) -- ({\i*\R+\j*\u+\k*\x+\w},{\i*\S+\j*\v+\k*\y+\z});
\foreach \i in {0,...,1}
    \foreach \j in {0,...,1}
        \foreach \k in {0,...,1}
            \foreach \h in {0,...,1}
                \draw node at ({\i*\r+\j*\u+\k*\x+\h*\w},{\i*\s+\j*\v+\k*\y+\h*\z}) {};
\foreach \i in {0,...,1}
    \foreach \j in {0,...,1}
        \foreach \k in {0,...,1}
            \foreach \h in {0,...,1}
                \draw node at ({\i*\R+\j*\u+\k*\x+\h*\w},{\i*\S+\j*\v+\k*\y+\h*\z}) {};
    \draw[line width=1.5pt, color=red] ({\r},{\s}) -- ({\r+\u},{\s+\v});
    \draw[line width=1.5pt, color=red] ({\r},{\s}) -- ({\r+\x},{\s+\y});
    \draw[line width=1.5pt, color=red] ({\r+\u},{\s+\v}) -- ({\r+\u+\w},{\s+\v+\z});
    \draw[line width=1.5pt, color=red] ({\r+\u},{\s+\v}) -- ({\r+\u+\x},{\s+\v+\y});
    \draw[line width=1.5pt, color=red] ({\r+\x},{\s+\y}) -- ({\r+\x+\w},{\s+\y+\z});
    \draw[line width=1.5pt, color=red] ({\r+\w},{\s+\z}) -- ({\r+\u+\w},{\s+\v+\z});
    \draw[line width=1.5pt, color=red] ({\r+\u+\x},{\s+\v+\y}) -- ({\r+\u+\x+\w},{\s+\v+\y+\z});
    \draw[line width=0.5pt] ({\u+\w},{\v+\z}) -- ({\r+\u+\w},{\s+\v+\z});
    \draw[line width=0.5pt] ({\u+\x},{\v+\y}) -- ({\r+\u+\x},{\s+\v+\y});
    \draw[line width=1.5pt, color=gren] ({0},{0}) -- ({\x},{\y});
    \draw[line width=1.5pt, color=gren] ({\w},{\z}) -- ({\x+\w},{\y+\z});
    \draw[line width=1.5pt, color=red] ({\u},{\v}) -- ({\u+\x},{\v+\y});
    \draw[line width=1.5pt, color=red] ({\u+\w},{\v+\z}) -- ({\u+\x+\w},{\v+\y+\z});
    \draw[line width=1.5pt, color=gren] ({0},{0}) -- ({\R},{\S});
    \draw[line width=0.5pt] ({\x+\w},{\y+\z}) -- ({\R+\x+\w},{\S+\y+\z});
    \draw[line width=1.5pt, color=gren] ({\R+\u+\w},{\S+\v+\z}) -- ({\R+\u+\x+\w},{\S+\v+\y+\z});
    \draw[line width=1.5pt, color=gren] ({\R+\x},{\S+\y}) -- ({\R+\x+\w},{\S+\y+\z});
    \draw[line width=1.5pt, color=gren] ({\R+\u},{\S+\v}) -- ({\R+\u+\x},{\S+\v+\y});
    \draw[line width=1.5pt, color=gren] ({\R+\u+\x},{\S+\v+\y}) -- ({\R+\u+\x+\w},{\S+\v+\y+\z});
    \draw[line width=1.5pt, color=gren] ({\R+\w},{\S+\z}) -- ({\R+\u+\w},{\S+\v+\z});
    \draw[line width=1.5pt, color=gren] ({\R+\x+\w},{\S+\y+\z}) -- ({\R+\u+\x+\w},{\S+\v+\y+\z});
\foreach \i in {0,...,1}
    \foreach \j in {0,...,1}
        \foreach \k in {0,...,1}
            \foreach \h in {0,...,1}
                \draw node at ({\i*\r+\j*\u+\k*\x+\h*\w},{\i*\s+\j*\v+\k*\y+\h*\z}) {};
\foreach \i in {0,...,1}
    \foreach \j in {0,...,1}
        \foreach \k in {0,...,1}
            \foreach \h in {0,...,1}
                \draw node at ({\i*\R+\j*\u+\k*\x+\h*\w},{\i*\S+\j*\v+\k*\y+\h*\z}) {};
\draw node [label=left: {\color{red} $a$}] at ({\r+\u},{\s+\v}) {};
\draw node [label=right: {\color{red} $a_1$}] at ({\u+\x},{\v+\y}) {};
\draw node [label=right: {\color{red} $a_2$}] at ({\u+\w},{\v+\z}) {};
\draw node [label=above: {\color{gren} $c$}] at ({\R+\u+\x+\w},{\S+\v+\y+\z}) {};
\draw node [label=above: {\color{gren} $c_1$}] at ({\x+\w},{\y+\z}) {};
\draw node [label=right: {\color{gren} $c_2$}] at ({\R},{\S}) {};
\end{tikzpicture}
$\qquad$ 
\begin{tikzpicture}[scale=.4]
\tikzstyle{every node}=[fill=white,minimum size=1pt,inner sep=2pt]
\def \s {2.5}
\draw node [label=center: {\underline{level}}] at (0,{6*\s}) {};
\foreach \i in {0,...,5}
    \pgfmathsetmacro{\l}{int(\i+9)}
    \draw node [label=center: {$\l$}] at (0,{\i*\s}) {};
\end{tikzpicture}
\caption{Coordinating the $\bz$-stacking of 3-cubes at levels 9--11 in $Q^{14}$.}
\label{fig:14Cubes}
\end{center}
\end{figure}
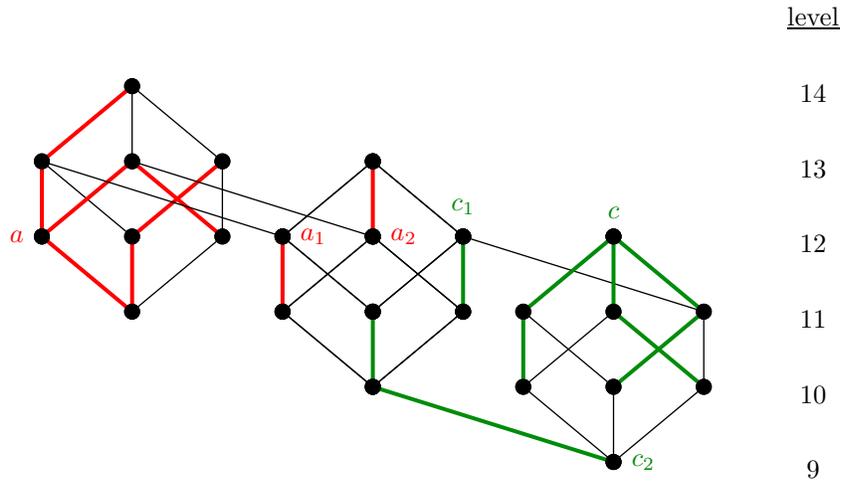

\begin{figure}[ht]
\begin{center}
\begin{tikzpicture}[scale=.4]
\tikzstyle{every node}=[draw,circle,fill=black,minimum size=1pt,inner sep=2pt]
\def \r {-8}
\def \R {8}
\def \s {2.5}
\def \S {-2.5}
\def \u {-3}
\def \v {2.5}
\def \x {0}
\def \y {2.5}
\def \w {3}
\def \z {2.5}
\foreach \i in {0,...,1}
    \foreach \k in {0,...,1}
        \foreach \h in {0,...,1}
            \draw[line width=0.5pt, color=black] ({\i*\r+0*\u+\k*\x+\h*\w},{\i*\s+0*\v+\k*\y+\h*\z}) -- ({\i*\r+1*\u+\k*\x+\h*\w},{\i*\s+1*\v+\k*\y+\h*\z});
\foreach \i in {0,...,1}
    \foreach \k in {0,...,1}
        \foreach \h in {0,...,1}
            \draw[line width=0.5pt, color=black] ({\i*\R+0*\u+\k*\x+\h*\w},{\i*\S+0*\v+\k*\y+\h*\z}) -- ({\i*\R+1*\u+\k*\x+\h*\w},{\i*\S+1*\v+\k*\y+\h*\z});
\foreach \i in {0,...,1}
    \foreach \j in {0,...,1}
        \foreach \h in {0,...,1}
            \draw[line width=0.5pt, color=black] ({\i*\r+\j*\u+0*\x+\h*\w},{\i*\s+\j*\v+0*\y+\h*\z}) -- ({\i*\r+\j*\u+\x+\h*\w},{\i*\s+\j*\v+\y+\h*\z});
\foreach \i in {0,...,1}
    \foreach \j in {0,...,1}
        \foreach \h in {0,...,1}
            \draw[line width=0.5pt, color=black] ({\i*\R+\j*\u+0*\x+\h*\w},{\i*\S+\j*\v+0*\y+\h*\z}) -- ({\i*\R+\j*\u+\x+\h*\w},{\i*\S+\j*\v+\y+\h*\z});
\foreach \i in {0,...,1}
    \foreach \j in {0,...,1}
        \foreach \k in {0,...,1}
            \draw[line width=0.5pt, color=black] ({\i*\r+\j*\u+\k*\x+0*\w},{\i*\s+\j*\v+\k*\y+0*\z}) -- ({\i*\r+\j*\u+\k*\x+\w},{\i*\s+\j*\v+\k*\y+\z});
\foreach \i in {0,...,1}
    \foreach \j in {0,...,1}
        \foreach \k in {0,...,1}
            \draw[line width=0.5pt, color=black] ({\i*\R+\j*\u+\k*\x+0*\w},{\i*\S+\j*\v+\k*\y+0*\z}) -- ({\i*\R+\j*\u+\k*\x+\w},{\i*\S+\j*\v+\k*\y+\z});
\foreach \i in {0,...,1}
    \foreach \j in {0,...,1}
        \foreach \k in {0,...,1}
            \foreach \h in {0,...,1}
                \draw node at ({\i*\r+\j*\u+\k*\x+\h*\w},{\i*\s+\j*\v+\k*\y+\h*\z}) {};
\foreach \i in {0,...,1}
    \foreach \j in {0,...,1}
        \foreach \k in {0,...,1}
            \foreach \h in {0,...,1}
                \draw node at ({\i*\R+\j*\u+\k*\x+\h*\w},{\i*\S+\j*\v+\k*\y+\h*\z}) {};
    \draw[line width=1.5pt, color=red] ({\r},{\s}) -- ({\r+\u},{\s+\v});
    \draw[line width=1.5pt, color=red] ({\r},{\s}) -- ({\r+\x},{\s+\y});
    \draw[line width=1.5pt, color=red] ({\r},{\s}) -- ({\r+\w},{\s+\z});
    \draw[line width=1.5pt, color=red] ({\r+\u},{\s+\v}) -- ({\r+\u+\x},{\s+\v+\y});
    \draw[line width=1.5pt, color=red] ({\r+\x},{\s+\y}) -- ({\r+\x+\w},{\s+\y+\z});
    \draw[line width=1.5pt, color=red] ({\r+\w},{\s+\z}) -- ({\r+\u+\w},{\s+\v+\z});
    \draw[line width=1.5pt, color=red] ({\r+\u+\x},{\s+\v+\y}) -- ({\r+\u+\x+\w},{\s+\v+\y+\z});
    \draw[line width=1.5pt, color=red] ({0},{0}) -- ({\r},{\s});
    \draw[line width=1.5pt, color=red] ({0},{0}) -- ({\u},{\v});
    \draw[line width=1.5pt, color=red] ({\u},{\v}) -- ({\u+\x},{\v+\y});
    \draw[line width=1.5pt, color=gren] ({\x},{\y}) -- ({\x+\w},{\y+\z});
    \draw[line width=1.5pt, color=gren] ({\w},{\z}) -- ({\u+\w},{\v+\z});
    \draw[line width=1.5pt, color=red] ({\u+\x},{\v+\y}) -- ({\u+\x+\w},{\v+\y+\z});
    \draw[line width=0.5pt] ({\x},{\y}) -- ({\R+\x},{\S+\y});
    \draw[line width=0.5pt] ({\w},{\z}) -- ({\R+\w},{\S+\z});
    \draw[line width=1.5pt, color=gren] ({\R},{\S}) -- ({\R+\x},{\S+\y});
    \draw[line width=1.5pt, color=gren] ({\R+\u},{\S+\v}) -- ({\R+\u+\x},{\S+\v+\y});
    \draw[line width=1.5pt, color=gren] ({\R+\x},{\S+\y}) -- ({\R+\x+\w},{\S+\y+\z});
    \draw[line width=1.5pt, color=gren] ({\R+\w},{\S+\z}) -- ({\R+\u+\w},{\S+\v+\z});
    \draw[line width=1.5pt, color=gren] ({\R+\w},{\S+\z}) -- ({\R+\x+\w},{\S+\y+\z});
    \draw[line width=1.5pt, color=gren] ({\R+\u+\x},{\S+\v+\y}) -- ({\R+\u+\x+\w},{\S+\v+\y+\z});
    \draw[line width=1.5pt, color=gren] ({\R+\x+\w},{\S+\y+\z}) -- ({\R+\u+\x+\w},{\S+\v+\y+\z});
\foreach \i in {0,...,1}
    \foreach \j in {0,...,1}
        \foreach \k in {0,...,1}
            \foreach \h in {0,...,1}
                \draw node at ({\i*\r+\j*\u+\k*\x+\h*\w},{\i*\s+\j*\v+\k*\y+\h*\z}) {};
\foreach \i in {0,...,1}
    \foreach \j in {0,...,1}
        \foreach \k in {0,...,1}
            \foreach \h in {0,...,1}
                \draw node at ({\i*\R+\j*\u+\k*\x+\h*\w},{\i*\S+\j*\v+\k*\y+\h*\z}) {};
\draw node [label=below: {\color{red} $a$}] at ({\r},{\s}) {};
\draw node [label=right: {\color{gren} $c$}] at ({\R+\x+\w},{\S+\y+\z}) {};
\draw node [label=left: {\color{gren} $c_1$}] at ({\x},{\y}) {};
\draw node [label=left: {\color{gren} $c_2$}] at ({\w},{\z}) {};
\end{tikzpicture}
$\qquad$ 
\begin{tikzpicture}[scale=.4]
\tikzstyle{every node}=[fill=white,minimum size=1pt,inner sep=2pt]
\def \s {2.5}
\draw node [label=center: {\underline{level}}] at (0,{6*\s}) {};
\foreach \i in {0,...,5}
    \pgfmathsetmacro{\l}{int(\i+10)}
    \draw node [label=center: {$\l$}] at (0,{\i*\s}) {};
\end{tikzpicture}
\caption{Coordinating the $\bz$-stacking of 3-cubes at levels 10--12 in $Q^{15}$.}
\label{fig:15Cubes}
\end{center}
\end{figure}
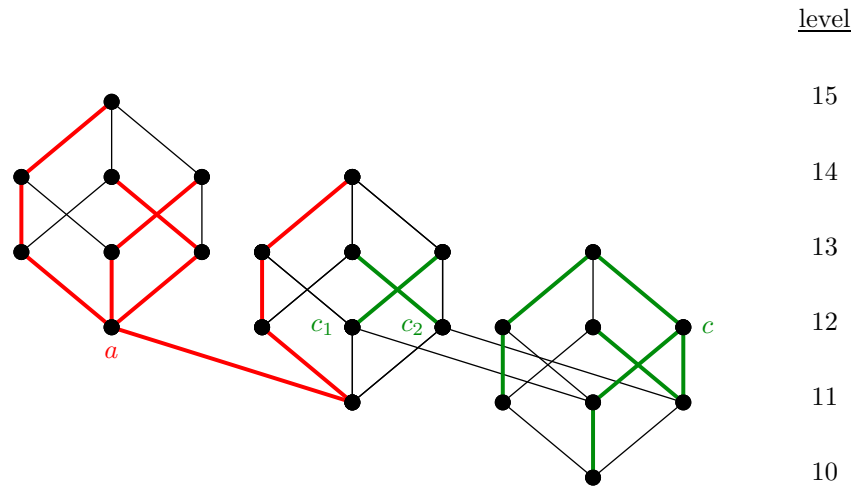

\begin{cor}
\label{c:15Cube}
For $d\le 15$, $Q^d$ is stackable.
\end{cor}

\begin{proof}
The cases $d\le 11$ is proven in Corollary \ref{c:11Cube}.

When $12\le d\le 15$ we write $G=Q^d=Q^{d-3}\ssq Q^3$, which partitions $Q^d$ into 3-cubes having levels at most $d-3$.
We choose $\f$ from Corollary \ref{c:Inject}, with $n=d-3$, 
For the highest level cube $A$ we define $B=\f(A)$ and $C=\f(B)$ and $\bz$-stack $A\cup B\cup C$ as in Lemma \ref{l:ABC3Cubes}.
After removing $A\cup B\cup C$ from $G$, the remaining graph $G'$, which is still partitioned into 3-cubes, has a highest level cube $A'$, from which we define $B'=\f(A')$ and $C'=\f(B')$ and $\bz$-stack $A'\cup B'\cup C'$ as in Lemma \ref{l:ABC3Cubes}.
We repeat this process until no cubes at level at least 9 remain; because $\f$ comes from a symmetric chain decomposition, each of the 3-cubes in this process are distinct.
i.e. the resulting graph is partitioned into 3-cubes having levels at most 8, the union of which is $\bz$-stackable by Lemmas \ref{l:LevelkCubes} and \ref{l:7Cube}.

This completes the proof.
\end{proof}

\begin{lem}
\label{l:4CubesLevels}
Let $k\le 15$, $2^k-k\le l\le 2^k$ and $d\ge l+k$.
Then every level-$l$ $k$-cube in $Q^d$ is $\bz$-stackable.
\end{lem}

\begin{proof}
Corollary \ref{c:15Cube} shows that $Q^k$ is stackable.
Therefore every level-$l$ $k$-cube in $Q^d$ can be stacked onto one of its weight-$2^k$ vertices, with the resulting $2^k$ cups then sent to $\bz$.
\end{proof}

\medskip
\noindent
{\it Proof of Theorem \ref{t:Cubes}.}
The cases $d\le 15$ are proven in Corollary \ref{c:15Cube}.

For $16\le d\le 20$ we write $Q^d=Q^{d-4}\ssq Q^4$, which partitions $Q^d$ into 4-cubes.
We use Lemma \ref{l:4CubesLevels} with $k=4$ to $\bz$-stack all 4-cubes at levels 12 and above.
After removing all such 4-cubes, the remaining graph $G$ is still partitioned into 4-cubes having levels at most 11.
Because $Q^4=Q^1\ssq Q^3$, we can more finely partition $G$ into 3-cubes having levels at most 12.
Now we choose $\f$ from Corollary \ref{c:Inject} with $n=d-3$ (from the representation $Q^d=Q^{d-3}\ssq Q^3$), and use Lemma \ref{l:ABC3Cubes} and the technique of Corollary \ref{c:15Cube} to $\bz$-stack $G$, finishing the proof.
\hfill$\Box$

\section{Conclusion}
\label{s:Conclusion}

One can imagine continuing an inductive approach to proving that $Q^d$ is stackable for all $d$, along the lines of the above.
However, one would need to conceive of a way of uniformizing the ad-hoc methods for each level of Lemma \ref{l:ABC3Cubes} in a manner that would readily generalize to higher dimensions.
Once achieved, generalizing Lemma \ref{l:4CubesLevels} to level-$l$ $k$-cubes with $2^k-l\le l\le 2^k$ would be straightforward.
For example, now that Theorem \ref{t:Cubes} has been proven, Lemma \ref{l:4CubesLevels} is true for $k\le 20$.
We are left believing that there is no mathematical obstruction to the stackability of all cubes and so make the following conjecture.

\begin{cnj}
\label{j:Cubes}
For all $d$, $Q^d$ is stackable.
\end{cnj}

Moreover, as $Q^d=\ssq_{i=1}^d P_2$, we believe the same to be true of all higher-dimensional grids, as shown by Theorem \ref{t:Grids} for dimension two.

\begin{cnj}
\label{j:Grids}
For all $d$ and $k_1,\ldots,k_d$, the grid $\ssq_{i=1}^d P_{k_i}$ is stackable.
\end{cnj}

The following enticing conjecture was posited by Veselovac, who verified it for $h\le 3$ and also for the root vertex only when $h\le 5$.

\begin{cnj} 
\cite{Vesel}
\label{j:btrees}
The complete binary tree $T_h$ with depth $h$ is stackable for all $h\ge 1$.
\end{cnj}

\subsection{Open Problems}
\label{ss:OpenProblems}

We mention here a small selection of interesting problems to consider.
Readers can likely think of several others.

\begin{enumerate}
    \item 
    Is there a characterization theorem for the stackability of a graph onto a vertex of eccentricity 3, akin to Lemma \ref{l:Char}?
    \item 
    Can stackable trees be characterized?
    For a general tree $T$, can the vertices $r$ for which $T$ is $r$-stackable be characterized?
    \item 
    Are cartesian products of stackable graphs stackable?
    \item
    Are generalized Johnson graphs of diameter at least 3 stackable?
    \item 
    Are there any non-trees that are not $r$-stackable for any $r$?\footnote{We thank one of the referees for inspiring this question.}
    \item
    What is the computational complexity for determining whether $G$ is $r$-stackable?
\end{enumerate}

\subsection{Variations}
\label{ss:Variations}

Additionally, one can imagine asking a similar range of questions about slightly modified versions of cup stacking.
One could:

\begin{itemize}
    \item 
    generalize the initial configuration --- we have seen in Figures \ref{fig:StackPart}, \ref{fig:11CubesP2}, and \ref{fig:11CubesP3} that this is a necessary generalization from the viewpoint of stacking the parts of an $r$-stacking partition;
    \item
    generalize the stacking target to a set $S$ of vertices --- one must stack onto any subset of $S$, onto exactly $S$, or onto a prescribed number of cups per vertex of $S$;
    \item
    allow for non-geodesic paths --- this would make every super-graph of a stackable graph stackable; for example, graphs with a Hamilton path; or
    \item
    allow for traversing trails or walks instead of just paths.
\end{itemize}

\section*{Acknowledgements}

We thank the referees for several suggestions that greatly improved the exposition of this paper.
In particular, we are grateful to one of the referees for pointing us to the prior work of Hamilton and Veselovac.


\bibliographystyle{acm}
\bibliography{refs}

\end{document}